\documentclass[reqno, a4paper]{amsart}
\usepackage{amsmath, amssymb, amsthm}

\pdfoutput=1
\usepackage{color}
\def\comment#1{}

\usepackage{hyperref}
\usepackage{txfonts}
\usepackage{enumerate}
\usepackage{caption}
\usepackage{accents}

\newtheorem{theorem}{Theorem}

\newtheorem{definition}[theorem]{Definition}
\newtheorem{lemma}[theorem]{Lemma}
\newtheorem{proposition}[theorem]{Proposition}

{Important Convention}

\theoremstyle{remark}
\newtheorem{remark}[theorem]{Remark}
\newtheorem{example}[theorem]{Example}

 \newcommand{\eps}{\varepsilon}

 \renewcommand{\phi}{\varphi}

\newcommand{\N}{\mathbb{N}}

\newcommand{\R}{\mathbb{R}}

\newcommand{\bes}{\begin{subequations}}
\newcommand{\ees}{\end{subequations}}
\newcommand{\eea}{\end{eqnarray}}

\usepackage{bbm}

\renewcommand{\eps}{\varepsilon}
\renewcommand{\epsilon}{\varepsilon}

\DeclareMathOperator{\proj}{proj}

\newcommand{\fourIdx}[5]{%
\setbox1=\hbox{\ensuremath{^{#1}}}%
 \setbox2=\hbox{\ensuremath{_{#2}}}%
 \setbox5=\hbox{\ensuremath{#5}}%
 \hspace{\ifnum\wd1>\wd2\wd1\else\wd2\fi}%
 \ensuremath{\copy5^{\hspace{-\wd1}\hspace{-\wd5}#1\hspace{\wd5}#3}%
 _{\hspace{-\wd2}\hspace{-\wd5}#2\hspace{\wd5}#4}%
 }}

\numberwithin{equation}{section}
\numberwithin{theorem}{section}

\renewcommand{\subset}{\subseteq}
\renewcommand{\supset}{\supseteq}

\usepackage{subcaption}
\usepackage{graphicx}
\renewcommand{\mathrm}{}

\newcommand{\probref}[1]{{\normalfont (\nameref{#1})}}

\begin{document}

\title{Existence, duality, and cyclical monotonicity for weak transport costs}
\author{J. Backhoff-Veraguas}
\author{M. Beiglb\"ock}
\author{G. Pammer}
\thanks{University of Vienna. Oskar-Morgenstern-Platz 1, 1090 Vienna, Austria.  MB gratefully acknowledges support by FWF-grant Y00782. GP acknowledges support from the Austrian Science Fund (FWF) through grant number W 1245. All authors thank the anonymous referee for insightful comments that lead to a significant improvement of the article.}  
\begin{abstract}
    The optimal weak transport problem has recently been introduced by Gozlan
    et.\ al.\ \cite{GoRoSaTe17}. We provide general existence and duality results for these problems on arbitrary Polish spaces, as well as a necessary and sufficient optimality 
    criterion in the spirit of cyclical monotonicity. As an application we
    extend the Brenier-Strassen Theorem of Gozlan-Juillet \cite{GoJu18} to general probability measures on $\R^d$ under minimal assumptions.
    
    {\color{black}A driving idea behind our proofs is to consider the set of transport plans with a new (`adapted') topology which seems better suited for the weak transport problem and allows to carry out arguments which are close to the proofs in the classical setup.}
    
    \medskip
    
{\color{black}\emph{Mathematics Subject Classification (2010):} 60G42, , 90C46, 58E30.}
\end{abstract}
\maketitle

\section{Introduction}
\subsection{Notation} 
This article is concerned with the optimal transport problem for weak costs, as initiated by Gozlan{\color{black}, Roberto, Samson and Tetali} in \ \cite{GoRoSaTe17}. To state it (see \eqref{eq weak transport def} below) we introduce some basic notation. 
 We write $\mathcal P(Z)$ for the set of probability measures on a Polish space $Z$ is   and equip $\mathcal P(Z)$ with the usual weak topology. %Denoting by $C_b(Z)$ the space of real-valued continuous bounded functions on $Z$, we use the probabilists terminology of `weak convergence' for the weak topology that $C_b(Z)$ induces on $\mathcal P(Z)$. 
 Throughout $X$ and $Y$ are Polish spaces, $\mu\in\mathcal P(X)$, and $\nu\in\mathcal P(Y)$. 
We write $\Pi(\mu,\nu)$ for the set of all couplings on $X\times Y$ with marginals $\mu$ and $\nu$.  Given a coupling $\pi$ on $X\times Y$ we denote a regular disintegration with respect to the first marginal by $(\pi_x)_{x\in X}$. %In this case, we endow $\mathcal P(X)$ and $\mathcal P(Y)$ with the topology of weak convergence. \\
We consider  cost functionals of the form 
$$C:X\times\mathcal P(Y)\to \mathbb R\cup\{+\infty\};$$
usually it is assumed  that $C$ is lower bounded and lower {\color{black}semicontinuous} in an appropriate sense, and that $C(x, \cdot)$ is convex. 
The weak transport problem is then defined as
\begin{align}\label{eq weak transport def}
V(\mu,\nu) := \inf_{\pi\in\Pi(\mu,\nu)}\int_X C(x,\pi_x)\mu(\mathrm dx).  
\end{align}

\subsection{Literature}
The initial works of Gozlan et al.\ \cite{GoRoSaTe17, GoRoSaSh18} 
are mainly motivated by applications to geometric inequalities. Indeed, particular costs of the form \eqref{eq weak transport def} were already considered by Marton \cite{Ma96concentration, Ma96contracting} and Talagrand \cite{Ta95, Ta96}. Further papers directly related to \cite{GoRoSaTe17} include \cite{Sh16, Sa17, Sh18, FaSh18, GoJu18}. 
Notably the weak transport problem \eqref{eq weak transport def} also yields a natural framework to investigate a number of related problems: it appears in the recursive formulation of the causal transport problem \cite{BaBeLiZa16}, in \cite{AlCoJo17, AlBoCh18, BeJu17, BaBeHuKa17} it is used to investigate  martingale optimal transport problems,  {\color{black}in \cite{BaBaBeEd19a} it is applied to prove stability of pricing and hedging in mathematical finance}, it appears in the characterization of optimal mechanism for the multiple good monopolist \cite{DaDeTz17} and motivates the investigation of linear transfers in \cite{BoGh18}.
A more classical example is given by entropy-regularized optimal transport (i.e.\ the  Schr\"odinger problem); see \cite{Le14} and the references therein.

\subsection{Main results}

We will establish analogues of three fundamental facts in optimal transport theory: existence of optimizers, duality, and characterization of optimizers through $c$-cyclical monotonicity. We make the important comment,  that these concepts (in particular existence and duality) have been previously studied for the weak transport problem. However, the results available so far may be too restrictive for certain applications.

 Our goal is to establish these results at a level of generality that mimics the framework usually considered in the optimal transport literature (i.e.\ lower bounded, lower semicontinuous cost function). We emphasize that this extension is in fact required to treat specific examples of interest, cf.\ Section \ref{sec:BrSt} below. 
 
We briefly hint at  the novel viewpoint which makes this extension possible: In a nutshell, the technicalities of the weak transport problem appear intricate and tedious since kernels $(\pi_x)_x$ are notoriously ill behaved with respect to weak convergence of measures on $\mathcal P(X\times Y)$. In the present paper we circumvent this difficulty by embedding $\mathcal P(X\times Y)$ into the bigger space $\mathcal P(X\times \mathcal P(Y))$. 
This idea is borrowed from the investigation of process distances (cf.\ \cite{PfPi12, BaBeEdPi17, BaBaBeEd19b}) and will allow us to carry out proofs that closely resemble familiar arguments from classical optimal transport.

\subsubsection{Primal Existence}
As a first contribution we will establish in Section 2 the following basic existence results.

\begin{theorem}[Existence I]\label{thm existence general} 
Assume that $C\colon X\times \mathcal P(Y)\rightarrow \mathbb{R}\cup\{+\infty\}$ is jointly lower semicontinuous, bounded from below, and convex in the second argument. Then,  the problem 
$$\inf_{\pi\in\Pi(\mu, \nu)}\int_X C(x,\pi_x)\mu(\mathrm dx),$$
admits a minimizer.
\end{theorem}

Notably, Gozlan et.al.\ {\color{black}prove existence of minimizers} under the assumption that $\pi\mapsto \int C(x, \pi_x)\, d\mu(x)$ is continuous on the set of all transport plans with first marginal $\mu$, whereas our aim is to establish existence based on properties of the function $C$. 
We also note that Theorem \ref{thm existence general} was first established by Alibert, Bouchitt\'e, and Champion {\color{black}in} \cite{AlBoCh18} in the case where $X,Y$ are compact spaces. 

In fact the assumptions of Theorem \ref{thm existence general} may be more restrictive than they initially appear. Indeed, as the cost function defined in \eqref{GJcost} below is not lower semicontinuous with respect to weak convergence, we will need to  employ  a refined version of Theorem \ref{thm existence general} to carry out our application in Theorem \ref{thm a la GJ} below. 

Given a compatible metric $d_Y$ on the Polish space $Y$ and $t\in[1,\infty)$, we write $\mathcal P_{d_Y}^t(Y)$ for the set of probability measures $\nu\in \mathcal P(Y)$ such that  $\int d_Y(y,y_0)^t\, \nu(dy)<\infty$ for some (and then any) $y_0\in Y$ and denote the $t$-Wasserstein metric on $\mathcal P_{d_Y}^t(Y)$ by $\mathcal W_t$ (see e.g.\ \cite[Chapter 7]{Vi03}). In the sequel we make the convention that, whenever we refer to $\mathcal P_{d_Y}^t(Y)$, it is assumed that this set is equipped with the topology generated by $\mathcal W_t$. On the other hand, regarding the Polish space $X$, we fix from now on a compatible bounded metric $d_X$.

\begin{theorem}[Existence II]\label{thm existence marginals} Assume that $\nu\in  \mathcal P^t_{d_Y}(Y)$.
Let $C\colon X\times \mathcal P_{d_Y}^t(Y)\rightarrow \mathbb{R}\cup\{+\infty\}$ be jointly lower semicontinuous with respect to the product topology on $X \times \mathcal P_{d_Y}^t(Y)$, bounded from below, and convex in the second argument.
Then,  the problem 
$$\inf_{\pi\in\Pi(\mu, \nu)}\int_X C(x,\pi_x)\mu( dx),$$
admits a minimizer.
\end{theorem}
We emphasize that Theorem \ref{thm existence general} is a special case of Theorem \ref{thm existence marginals}. To see this, just take $d_Y$  to be a compatible bounded metric. 
%Moreover, Theorem \ref{thm existence  marginals} will be derived as a  consequence of Theorem \ref{thm existence compactset} below, in which the marginals are not necessarily fixed.
We also note that if $C$ is strictly convex in the second argument and $V(\mu,\nu)<\infty$, then the minimizer $\pi^*\in \Pi(\mu,\nu)$ is unique. We report our proofs in Section \ref{sec existence}.

\medskip

\subsubsection{Duality}

We fix a compatible metric  $d_Y$ on $Y$ and introduce the space
\begin{align}\label{def:intro Phi}
\Phi_{b,t} :=\{\psi:Y\to\mathbb R \text{ {\color{black}continuous}  s.t. $\exists a,b,\ell\in\mathbb R, y_0\in Y,$ }\ell\leq\psi(\cdot)\leq a+bd_Y(y_0,\cdot)^t\}
.
\end{align}
To each $\psi\in \Phi_{b,t}$ we associate the function
\begin{align}\label{def:RCphi}
	R_C\psi(x): = \inf_{p \in \mathcal P_{d_Y}^t(Y)} p(\psi) + C(x,p).
\end{align}
{We remark that $R_C\psi(\cdot)$ is universally  measurable if $C$ is measurable (\cite[Proposition 7.47]{BeSh78}) and so the integral $\mu(R_C\psi)$ is well defined for all $\mu\in\mathcal P(Y)$ if $C$ is lower-bounded. 
\begin{theorem}\label{th:kantorovichduality}
	Let $C\colon X\times \mathcal P_{d_Y}^t(Y)\rightarrow \mathbb{R}\cup\{\infty\}$ be jointly lower semicontinuous with respect to the product topology on $X \times \mathcal P_{d_Y}^t(Y)$, bounded from below, and convex in the second argument. Then we have {\color{black}for all} $\mu\in\mathcal P(X)$ and $\nu\in\mathcal P_{d_Y}^t(Y)$
	\begin{align}
		V(\mu,\nu) = \sup_{\psi \in \Phi_{b,t}} \mu(R_C\psi) - \nu(\psi).
	\end{align}
\end{theorem}
The proof of Theorem \ref{th:kantorovichduality} is provided in Section \ref{sec:duality}. We also refer to this section for a comparison of earlier duality results of Gozlan, Roberto, Samson, and Tetali 
\cite[Theorem 9.6]{GoRoSaTe17} and Alibert, Bouchitt\'e, and Champion \cite[Theorem 4.2]{AlBoCh18}.

\medskip

\subsubsection{$C$-monotonicity}
Besides primal existence and duality, another fundamental result in classical optimal transport is the characterization of optimality through the notion of \emph{cyclical monot{\color{black}o}nicity}; see \cite{RaRu90, GaMc96} as well as the monographs \cite{RaRu98, Vi03, Vi09}. More recently, variants of this `monotonicity priniciple' have been applied in transport problems for finitely or infinitely many marginals \cite{Pa12fm, CoDeDi15, Gr16a, BeGr14, Za14}, the martingale version of the optimal transport problem \cite{BeJu16, NuSt16, BeNuTo16}, the Skorokhod embedding problem \cite{BeCoHu14} and the distribution constrained optimal stopping problem \cite{BeEdElSc16}.

 We provide in Definition \ref{def:C-monotonicity} below, a  concept analogous to cyclical monotonicity (which we call $C$-monotonicity) for weak transport costs $C$. We show that every optimal transport plan is $C$-monotone in a very general setup. Conversely, we {\color{black}have} that every $C$-monotone transport plan is optimal under certain regularity assumptions. See Theorems \ref{thm cyclical} and \ref{th:compactsufficient} respectively.

We note that related concepts already appeared in \cite[Proposition 4.1]{BaBeHuKa17} (where necessity of a 2-step optimality condition is established) and in \cite{GoJu18} (necessity in the case of compactly supported measures and a quadratic cost criterion). To the best of our knowledge, our sufficient criterion is the first of its kind for weak transport costs.

{\color{black}We remark that the 2-step monotonicity principle for weak transport costs has already proved vital in \cite{BaBeHuKa17} for the construction of a martingale counterpart to the Brenier theorem and the Benamou-Brenier formula. On the other hand, we conjecture that this monotonicity principle could be used in order to generalize \cite{GoJu18} to non-quadratic costs. }
\medskip

\subsubsection{A general Brenier-Strassen theorem}\label{sec:BrSt}
As an application of our abstract results we extend the Brenier-Strassen theorem  \cite[Theorem 1.2]{GoJu18} of Gozlan and Juillet  to the case of general probabilities on $X=Y=\R^d$ under the assumption that $\mu$ has finite second moment and $\nu$ has finite first moment. We thus drop the condition in \cite{GoJu18} that the marginals have compact support. For this part we set \begin{align}\label{GJcost}C(x,\rho):= \left| x - \int y\rho(dy) \right|^2,\end{align}
and write $\leq_c$ for the convex order of probability measures.

\begin{theorem}\label{thm a la GJ}
Let $\mu\in \mathcal P^2(\mathbb R^d)$ and $\nu\in \mathcal P^1(\mathbb R^d)$. There exists a unique $\mu^* \leq_c\nu$ such that
\begin{equation}\label{eq GJ mu star}
\mathcal W_2(\mu^*,\mu)^2= \inf_{\eta\leq_c\nu}\mathcal  W_2(\eta,\mu)^2=V(\mu,\nu).
\end{equation}
Moreover, there exists a convex function $\phi:\mathbb R^d\to\mathbb R$ of class $\mathcal C^1$ with $\nabla \phi$ being 1-Lipschitz, such that $\nabla\phi(\mu)=\mu^*$. Finally, an optimal coupling $\pi^*\in\Pi(\mu,\nu)$ for $V(\mu,\nu)$ exists, and a coupling $\pi\in\Pi(\mu,\nu)$ is optimal for $V(\mu,\nu)$ if and only if $\int y\pi_x(dy)=\nabla\phi(x)$ $\mu$-a.s. 
\end{theorem}

Existence of $\mu^*$ and the expression \eqref{eq GJ mu star} were first proved by Gozlan et al \cite{GoRoSaSh18} for $d=1$ and by Alfonsi, Corbetta, Jourdain \cite{AlCoJo17} for arbitrary $d\in\N$. Indeed a general version of \eqref{eq GJ mu star}, appealing to $\mathcal W_p$ and probabilities $\mu, \nu \in \mathcal P^p(\mathbb R^d)$ is provided in \cite{AlCoJo17}.  All other statements in the above theorem were originally established by Gozlan and Juillet \cite{GoJu18} under the assumption of compactly supported measures $\mu, \nu$. The proof of Theorem \ref{thm a la GJ} is given in Section \ref{sec BS GJ}.
{\color{black}

{\bf Note added in revision.} In an updated version of \cite{GoJu18}, Gozlan and Juillet have also removed the compactness assumption in Theorem~\ref{thm a la GJ}. Their proof is based on duality arguments and in particular differs from the one given here.}

\section{Existence of minimizers }\label{sec existence}

{\color{black}A principal idea behind the proofs of this paper is to endow the set of transport plans $\mathcal P(X\times Y)$ with a topology that is finer than the usual weak topology and which appropriately accounts for the asymmetric role of $X$ and $Y$ in the context of weak transport. This can be formalized by embedding $\mathcal P(X\times Y)$ into the bigger space $\mathcal P(X\times \mathcal P(Y))$. I.e., given a transport plan $\pi$, we will consider its disintegration $(\pi_x)_{x\in X}$ (w.r.t. its first marginal) and view it as a \emph{Monge-type} coupling in the larger space $\mathcal P(X\times \mathcal P(Y))$.} It turns out that on this `extended' space the minimization problems  Theorem~\ref{thm existence marginals} and Theorem~\ref{thm existence general} can be handled more {\color{black}efficiently}.

We need to introduce additional notation: for a probability measure $\pi\in \mathcal P(X\times Y)$ with not further specified marginals, we write $\pi(\mathrm dx\times Y)$ and $\pi(X\times \mathrm dy)$ for its $X$-marginal and $Y$-marginal respectively. At several instances we use the projection from a product space onto one of its components. This map is usually denoted by $\proj_\bullet$ where the subscript describes the component, e.g.\ $\proj_X\colon X\times Y \rightarrow X$ stands for the projection onto the $X$-component.
Denoting by $(\pi_x)_{x\in X}$ a regular disintegration of $\pi$ with respect to $\pi(dx\times Y$), we  consider the measurable map 
\begin{align*}
\kappa_\pi \colon  &X \rightarrow X \times \mathcal P(Y)\\ & \,x\mapsto (x,\pi_x).
\end{align*}
We define the embedding $J\colon  \mathcal P(X\times Y ) \rightarrow \mathcal P(X\times\mathcal P(Y))$ by setting for $\pi\in  \mathcal P(X\times Y ) $ with $X$-marginal $\mu(dx)=\pi(dx\times Y)$
\begin{align}\label{eq:def J}
J(\pi):=  (\kappa_\pi)_\#(\mu).
\end{align}
The map $J$ is well-defined since $\kappa_\pi$ is $\pi(dx\times Y)$-almost surely unique. Note that elements in $\mathcal P(X\times Y)$ precisely correspond to those elements of $\mathcal P(X\times \mathcal P(Y))$ which are concentrated on a graph of a measurable function from $X$ to $\mathcal P(Y)$.

The intensity $I(P)\in\mathcal P(Y)$ of $P\in\mathcal P(\mathcal P(Y))$ is uniquely determined by
\begin{align}\label{eq:def intensity1}
	I(P)(f) = \int_{\mathcal P(Y)} p(f) P(dp) \quad \forall f\in C_b(Y).
\end{align}
The set of all probability measures $P\in\mathcal P(X \times \mathcal P(Y))$ with $X$-marginal $\mu$ and `{\color{black}$\mathcal P(Y)$}-marginal intensity' $\nu$ is denoted by
\begin{align}
	\Lambda(\mu,\nu) := \left\{ P \in \mathcal P(X\times \mathcal P(Y))\, \mid\, \proj_X P = \mu,\, I(\proj_{\mathcal P(Y)}(P)) = \nu\right\}.
\end{align}
Similar to \eqref{eq:def intensity1}, we define the intensity of $P\in \mathcal P(X\times \mathcal P(Y))$ as the unique measure $\hat I(P) \in \mathcal P(X\times Y)$ such that
\begin{align}\label{eq:def intensity2}
\int_{X\times Y} f(x,y) \hat I(P)(dx,dy) = \int_{X\times \mathcal P(Y)} \int_Y f(x,y)p(dy) P(dx,dp)\quad \forall f\in C_b(X\times Y).
\end{align}
Note that while $J$ is in general not continuous (cf.\ Example \ref{ex:J discontinuous}), the mappings $I$ and $\hat I$ are continuous.

Using \eqref{eq:def J} and \eqref{eq:def intensity2} we find that $$\Lambda(\mu,\nu) = \hat I^{-1}(\Pi(\mu,\nu))\text{  and  }J(\Pi(\mu,\nu))\subset \Lambda(\mu,\nu).$$
{\color{black}Also note that $\hat I$ is the left-inverse of $J$, i.e., $\hat I \circ J(\pi) = \pi$ for $\pi \in \mathcal P(X\times Y)$.}
We now describe the relation between minimization problems on $\Pi(\mu,\nu)$ and $\Lambda(\mu,\nu)$:

\begin{lemma}\label{lem:convex optimality equivalence}
	Let $C\colon X\times \mathcal P(Y) \rightarrow \R\cup\{-\infty,+\infty\}$ be measurable, {\color{black}lower-bounded}, and convex in the second argument. Then
	\begin{align}
		V(\mu,\nu) = \hat V(\mu,\nu),
	\end{align}
	where $V$ was defined in \eqref{eq weak transport def} and
	\begin{align}\label{eq:defhatV}
	\hat V(\mu,\nu):=\inf_{P \in \Lambda(\mu,\nu)} \int_{X\times \mathcal P(Y)} C(x,p)P(dx,dp).
	\end{align}
\end{lemma}

\begin{proof}
	For any $\pi\in\Pi(\mu,\nu)$ we have $J(\pi) \in \Lambda(\mu,\nu)$ and
	$$\int_X C(x,\pi_x)\mu(dx) = \int_{X\times \mathcal P(Y)} C(x,p) J(\pi)(dx,dp).$$
	Thus,
	$$\inf_{\pi\in\Pi(\mu,\nu)} \int_X C(x,\pi_x) \mu(dx) \geq \inf_{P\in\Lambda(\mu,\nu)} \int_{X\times \mathcal P(Y)} C(x,p)P(dx,dp).$$
	Now, letting $P\in\Lambda(\mu,\nu)$, we easily derive from \eqref{eq:def intensity2} that $\hat I(P) \in \Pi(\mu,\nu)$ and $\hat I(P)_x = \int_{\mathcal P(Y)} p \,P_x(dp)$ for $\mu$-a.e $x$. Using convexity we conclude
	\begin{align*}
	\int_{X\times \mathcal P(Y)} C(x,p) P(dx,dp) &=
	\int_X \int_{\mathcal P(Y)} C(x,p) P_x(dp) \mu(dx) \\
	&\geq \int_X C(x,\hat I(P)_x) \mu(dx)\\
	&\geq \inf_{\pi\in\Pi(\mu,\nu)} \int_X C(x,\pi_x) \mu(dx).
	\end{align*}
\end{proof}

\subsection{Existence of minimizers}
The purpose of this subsection is to establish Theorem~\ref{thm existence marginals}, or more precisely, a strengthened version of it; see Theorem~\ref{thm existence compactset} below. To this end we need a number of auxiliary results.

We start by stressing that, in general, the embedding $J$ is not continuous. In fact: 

\begin{example} \label{ex:J discontinuous}
The map $J$ is continuous if and only if $X$ is discrete or $|Y|=1$. Indeed, given $X$ discrete and a sequence $(\pi^k)_{k\in\N}\in\mathcal P(X\times Y)^\N$ which weakly converges to $\pi$, we have that $ \pi^k(x\times Y)\to \pi(x \times Y)$ from which $\pi_x^k(dy)=\frac{\pi^k(x,dy)}{\pi^k(x\times Y)}$ converges weakly to $\pi_x(dy)=\frac{\pi(x,dy)}{\pi(x\times Y)}$ if $\pi(x \times Y)>0$.
%$$\lim_k \pi^k(x\times Y) = \pi(x\times Y),\quad \pi^k_x\pi^k(x\times Y) \rightharpoonup \pi_x\pi(x\times Y) \text{ for }k\rightarrow \infty.$$
Consequently if $f \in C_b(X\times \mathcal P(Y))$, then
\begin{align*}
\lim_k &|J(\pi^k)(f) - J(\pi)(f)| \\ &\leq \limsup_k \sum_x |f(x,\pi^k_x)(\pi^k(x\times Y)-\pi(x\times Y))| + \sum_x|f(x,\pi^k_x) - f(x,\pi_x)|\pi(x\times Y) \\ &= 0.
\end{align*}
Therefore $(J(\pi^k))_{k\in\N}$ converges weakly to $J(\pi)$.  On the other hand, suppose there is a sequence $(x_k)_{k\in\N}\in X^\N$ of distinct points converging to some $x\in X$, as well as $p,q\in\mathcal P(Y)$ with $p\neq q$. For $k\in\N$ define a probability measure on $\mathcal P(X\times Y)$ by
$$\pi^k(dx,dy) := \frac{1}{2}(\delta_{x_{k+1}}(dx)p(dy) + \delta_{x_k}(dx)q(dy)).$$
A short computation yields
\begin{align*}
	&\lim_kJ(\pi^k)= \lim_k\frac{1}{2}\left(\delta_{(x_{k+1},p)} + \delta_{(x_k,q)}\right) =\frac{1}{2}\left(\delta_{(x,p)} + \delta_{(x,q)}\right),\\
	&J(\lim_k \pi^k)= J\Big(\frac{1}{2}\delta_x(p+q)\Big) = \delta_{\left(x,\frac{1}{2}(p+q)\right )},
\end{align*}
which shows that $J$ is discontinuous.

\end{example} 

On the bright side, $J$ possesses a crucial feature: it maps relatively compact sets to relatively compact sets. We prove this in Lemma \ref{lem:relcompembedding} below. But first we need to digress into the characterization of tightness on $\mathcal P(\mathcal P(Y))$ and subspaces thereof. The following can be found in \cite[p. 178, Ch. II]{Sz91}.
\begin{lemma}\label{th:tightness_old}
	A set $\mathcal A\subseteq \mathcal P(\mathcal P(Y))$ is tight if and only
	if the set of its intensities $I(\mathcal A)$ is tight in $\mathcal P(Y)$.
\end{lemma}

We need to refine Lemma~\ref{th:tightness_old} for our purposes, since we equip $\mathcal P_{d_Y}^t(Y)$ with the $\mathcal W_t$-topology instead of the weak topology.
\begin{lemma}\label{th:tightness_new}
	A set $\mathcal A\subseteq\mathcal P^t_{\mathcal W_t}(\mathcal P_{d_Y}^t(Y))$ is relatively compact if and only
	if the set of its intensities $I(\mathcal A)$ is relatively compact in $\mathcal P_{d_Y}^t(Y)$.
\end{lemma}

The proof of Lemma~\ref{th:tightness_new} heavily relies on the following lemma, for which we 
include a proof for sake of completeness.
\begin{lemma}\label{lem:wasserstein_relcomp}
	A set $\mathcal A\subseteq \mathcal P_{d_Y}^t(Y)$ is relatively compact if and only if it is
	tight and 
	\begin{align}\label{eq:wasserstein_relcomp}
	\exists y'\in Y\,\forall \epsilon > 0 ~\exists R>0\colon\sup_{\mu\in \mathcal A}\int_{B_R(y')^c} d_Y(y,y')^t\mu( dy) < \epsilon.
	\end{align}
\end{lemma}
Note that if \eqref{eq:wasserstein_relcomp} holds for some $y'\in Y$ it automatically holds for any $y' \in Y$.

\begin{proof}[Proof of Lemma~\ref{th:tightness_new}]
%The first implication follows by continuity of $I$ and Lemma~\ref{th:tightness_old} provides tightness.
{\color{black} Since continuous maps preserve relative compactness in Hausdorff spaces, the first implication follows by continuity of $I$. To show the reverse implication, let $I(\mathcal A)$ be} relatively compact in $\mathcal P_{d_Y}^t(Y)$. First we show for fixed $y'\in Y$ that
\begin{align}\label{eq:JulioYellow}\forall \epsilon>0\, \exists R_\epsilon > 0 \colon \sup_{P\in\mathcal A} \int_{\{p\colon \mathcal W_t(p,\delta_{y'})^t \geq 
R_\epsilon\}} \mathcal W_t(p,\delta_{y'})^tP( dp) <\epsilon.\end{align} 
Fix $\epsilon > 0$. There exist $K>0$ and $r>0$ such that for all $P\in \mathcal A$
\begin{align} \int_{\mathcal P^t_{d_Y}(Y)} \mathcal W_t(p,\delta_{y'})^tP( dp)&=\int_Y d_Y(y,y')^tI(P)( dy)\leq K \notag \\
\int_{\mathcal P^t_{d_Y}(Y)} \int_{B_r(y')^c}d_Y(y,y')^tp( dy)P( dp) & = \int_{B_r(y')^c} d_Y(y,y')^tI(P)( dy) < \frac{\epsilon}{2}, \label{eq:lemma moment estimate1}
\end{align}
where $B_r(y') = \{y\in Y \colon d_Y(y,y') < r\}$.
Set $R_\epsilon = \frac{2r^tK}{\epsilon}$ and $A_{R_\epsilon}= \left\{p\in\mathcal P_{d_Y}^t(Y)\colon\mathcal W_t(p,\delta_{y'})^t\geq R_\epsilon\right\}$, then
$$\sup_{P\in\mathcal A} P(A_{R_\epsilon}) \leq \sup_{P\in\mathcal A}\frac{1}{R_\epsilon} \int_{A_{R_\epsilon}} \mathcal W_t(p,\delta_{y'})^tP(dp)\leq \frac{K}{R_\epsilon}$$
and
\begin{align}\label{eq:lemma moment estimate2}
\sup_{P\in \mathcal A} \int_{A_{R_\epsilon}} \int_{B_r(y')}d_Y(y,y')^t p( dy)P(dp)\leq \sup_{P\in\mathcal A}P(A_{R_\epsilon}) r^t \leq \frac{\epsilon}{2}.
\end{align}
Putting \eqref{eq:lemma moment estimate1} and \eqref{eq:lemma moment estimate2} together shows \eqref{eq:JulioYellow}. 

It remains to show that $\mathcal A$ is tight in $\mathcal P(\mathcal P_{d_Y}^t(Y))$. By Lemma~\ref{th:tightness_old} we have that $\mathcal A$ is tight in $\mathcal P(\mathcal P(Y))$, i.e., given $\epsilon > 0$ there is a compact set $K_\epsilon \subset \mathcal P(Y)$ such that for all $P\in \mathcal A$ we have $P(K_\epsilon)\geq 1 - \epsilon$. We will construct a set $\tilde K_\epsilon \subset K_\epsilon$ which is compact in $ \mathcal P_{d_Y}^t(Y)$ and satisfies $P(\tilde K_\epsilon)\geq 1 - 2\epsilon$ in $P\in \mathcal A$.   To this end,  take a sequence of  radii $(R_n)_{n\in\N}$ such that 
\begin{align*}
	\sup_{P \in \mathcal A} P\left(\left\{p\colon \int_{\{y \colon d_Y(y,y')^t > R_n\}} d_Y(y,y')^t p(dy) \geq \frac{1}{n} \right\}\right) < \frac{\epsilon}{2^n},
\end{align*}
which is possible since
\begin{align*}%\textstyle
	P\left(\left\{p \colon \int_{\{y \colon d_Y(y,y')^t > R_n\}} d_Y(y,y')^t p(dy) \geq \frac{1}{n} \right\}\right) \leq n \int_{\{y\colon d_Y(y,y')^t > R_n\}} d_Y(y,y')^t I(P)(dy),
\end{align*}
can be chosen sufficiently small, uniformly for $P\in\mathcal A$. The set
\begin{align*}
	\tilde K_\epsilon := \left\{p \in K_\epsilon\colon \mathcal \int_{\{y\colon d_Y(y,y')^t > R_n\}} d_Y(y,y')^t p(dy) \leq \frac{1}{n},\quad n\in\N \right\}
\end{align*}
is compact in $\mathcal P_{d_Y}^t(Y)$ (cf.\ Lemma \ref{lem:wasserstein_relcomp}).
Finally, given $P\in \mathcal A$ we obtain
\begin{align*}
	P(\tilde K_\epsilon) &\geq P(K_\epsilon) - \sum_n P\left(\left\{p\colon \int_{\{y\colon d_Y(y,y')^t > R_n\}} d_Y(y,y')p(dy) \geq \frac{1}{n}\right\}\right) \geq 1 - 2\epsilon
\end{align*}
as desired
\end{proof}
\medskip

\begin{proof}[Proof of Lemma~\ref{lem:wasserstein_relcomp}]~\par

	`$\Rightarrow$':
	Since the topology induced by $\mathcal W_t$ on $\mathcal P_{d_Y}^t(Y)$ is finer than the weak topology on 
	$\mathcal P_{d_Y}^t(Y)$, relative compactness in $\mathcal W_t$ implies relative
	compactness with respect to the weak topology. Therefore, Prokhorov's theorem yields tightness.
	 Suppose for contradiction that \eqref{eq:wasserstein_relcomp} fails, i.e. there exist $y' \in Y$ and $\epsilon > 0$ such that for all $N\in\N$ there is $\mu_N \in \mathcal A$ s.t. $$\int_{B_N(y')^c} d_Y(y,y')^t\mu_N(dy) \geq \epsilon.$$
	In particular, 
	\begin{align} \label{eq:contradiction relcomp}
	\lim_{R\to \infty} \liminf_N \int_{B_R(y')^c} d_Y(y',y)^t \mu_N(dy) \geq \epsilon.
	\end{align}	
	 Due to relative compactness we find for any sequence in $\mathcal A$ an accumulation point. Then, from the definition of $\mathcal W_t$-convergence, see \cite[Definition 6.8 $(iii)$]{Vi09}, we deduce
	 $$\lim_{R\to\infty} \liminf_N \int_{B_R(y')^c} d_Y(y',y)^t \mu_N(dy) = 0,$$
	 which contradicts \eqref{eq:contradiction relcomp}. Hence, \eqref{eq:wasserstein_relcomp} is satisfied.
	%Note that \eqref{eq:wasserstein_relcomp} follows immediately from the definition of
	%convergence in $\mathcal W_t$.
	
	\medskip
	
`$\Leftarrow$':	Let $\mathcal A$ be tight such that \eqref{eq:wasserstein_relcomp} holds. Then, any sequence $(\mu_k)_{k\in\mathbb{N}}\in \mathcal A^\mathbb{N}$ has an accumulation point $\mu\in\mathcal P(Y)$ with respect to the weak topology. Without loss of generality
	assume that $\mu_k \to \mu$ for $k\rightarrow \infty$.
	By monotone convergence %and the Portmanteau theorem, we have for $O\subset Y$ open
	\begin{align*}
	\int d_Y(y,y')^t\mu( dy) &=\lim_{R\rightarrow \infty} \int R\wedge d_Y(y,y')^t\mu( dy) \\ & \leq\lim_{R\rightarrow \infty}\liminf_{n\to \infty} \int R\wedge d_Y(y,y')^t\mu_n( dy)  \leq \sup_n \int  d_Y(y,y')^t\mu_n( dy).
	\end{align*}		
	Hence, by \eqref{eq:wasserstein_relcomp}
	we can choose (for $\epsilon=1 $, say)  $R>0$ such that
	$$\int_Y d_Y(y,y')^t \mu( dy) \leq \sup_n \int_{B_R(y')} d_Y(y,y')^t \mu_n( dy)
	+ 1 < \infty,$$
	which shows that $\mu\in\mathcal P_{d_Y}^t(Y)$.

	Next, fix $\epsilon>0$. Pick $R>0$  such that
	\begin{align*}\int_Y d_Y(y,y')^t-R^t\wedge d_Y(y,y')^t\mu( dy)& <\eps,   \\ \ 
	\sup_{n}\int_{B_R(y')^c} d_Y(y,y')^t\mu_n( dy) & < \epsilon.
	\end{align*}
	By weak convergence we know that 	$$\lim_k \int_Y R^t\wedge d_Y(y,y')^t \mu_k( dy)\rightarrow 
	\int_Y R^t\wedge d_Y(y,y')^t\mu( dy).$$
Hence we may pick $k_0$ such that for all $k \geq k_0$ 
	$$\Big|\int_Y R^t\wedge d_Y(y,y')^t(\mu_k-\mu)(dy)\Big|<{\epsilon}.$$
	Thus we have for $k\geq k_0$ 
	$$\Big|\int_Y d_Y(y,y')^t(\mu_k - \mu)(dy)\Big| < 3 \epsilon.$$
	Since $\epsilon$ was arbitrary, we obtain that the $t$-moments are converging, which implies
	convergence in $\mathcal W_t$.
\end{proof}

We recall that on $Y$ we are usually given a compatible complete metric $d_Y$,  whereas on $X$ we fix a compatible bounded metric $d_X$. We thus endow the product spaces $X\times Y$ and $X\times \mathcal P_{d_Y}^t(Y)$ with natural (product) metrices $d$ and $\hat d$ defined respectively by
\begin{align}\label{eq:def metric XxY}
	d((x,y),(x_0,y_0)) &= d_X(x,x_0) + d_Y(y,y_0),\\ \label{eq:def metric XxPY}
	\hat d((x,p),(x_0,p_0)) &= d_X(x,x_0) + \mathcal W_t(p,p_0).
\end{align}
We can now state and prove the crucial property of $J$:

\begin{lemma}\label{lem:relcompembedding}
If $\Pi\subseteq \mathcal P_{d}^t(X\times Y)$ is relatively compact then $J(\Pi)\subseteq \mathcal P^t_{\hat d}(X\times \mathcal P_{d_Y}^t(Y))$ is relatively compact. {\color{black}Conversely, if $\Lambda \in \mathcal P^t_{\hat d}(X\times \mathcal P_{d_Y}^t(Y))$ is relatively compact then $\hat I(\Lambda)\subseteq \mathcal P^t_d(X\times Y)$ is relatively compact.}
\end{lemma}

\begin{proof}
{\color{black}Since continuous maps preserve relative compactness in Hausdorff spaces, we immediately deduce relative compactness of $\hat I(\Lambda)$, and the sets $\Pi^X\subset \mathcal P(X)$ and $\Pi^Y\subset \mathcal P_{d_Y}^t(Y)$ consisting respectively of the $X$- and $Y$-marginals of the elements in $\Pi$.}
%By continuous mapping (see \cite[Theorem A.3.10]{DuEl11}) the sets $\Pi^X\subset \mathcal P(X)$ and $\Pi^Y\subset \mathcal P_{d_Y}^t(Y)$, consisting respectively of the $X$- and $Y$-marginals of the elements in $\Pi$, are tight. Then, relative compactness of $\Pi^X$ and $\Pi^Y$ can be readily derived by courtesy of Lemma~\ref{lem:wasserstein_relcomp} and the structure of the product metric $d$, cf.~\eqref{eq:def metric XxY}.

Denote now respectively by $\Pi_J^X\subset \mathcal P(X)$ and $\Pi_J^Y\subset \mathcal P_{\mathcal W_t}^t(\mathcal P_{d_Y}^t(Y))$ the set of $X$- and $\mathcal P(Y)$-marginals of the elements in $J(\Pi)$. Clearly $\Pi_J^X=\Pi^X$. By Lemma~\ref{th:tightness_new}, the set $\Pi_J^Y$ is relatively compact in $\mathcal P_{\mathcal W_t}^t(\mathcal P_{d_Y}^t(Y))$ if and only if the set {\color{black}$I(\Pi_J^Y)$} is relatively compact in $\mathcal P_{d_Y}^t(Y)$. However, if $m$ is equal to the $\mathcal P(Y)$-marginal of $J(\pi)$, then  $I(m)$ is equal to the $Y$-marginal of $\pi$. It follows that $I(\Pi_J^Y)\subset \Pi^Y$ is relatively compact and so is $\Pi_J^Y$. Since the marginals of $J(\Pi)$ are relatively compact, we conclude that $J(\Pi)$ itself is relatively compact.
\end{proof}

It is convenient to introduce the following assumptions, which we will often require: 

\begin{definition}[\textsc{A}]\label{def:conditionA}
	Given Polish spaces $X$, $Y$, we say that a function 
	$$C\colon X\times\mathcal P_{d_Y}^t(Y)	\rightarrow \mathbb{R}\cup\{+\infty\}$$ satisfies Condition~\probref{def:conditionA} if the following hold:
	\begin{itemize}
		\item  $C$ is lower semicontinuous with respect to the product topology of $$(X,d_X) \times (\mathcal P^t_{d_Y}(Y), \mathcal W_t),$$
		\item $C$ is bounded from below.
	\end{itemize}
	If in addition for all $x\in X$ the map $p \mapsto C(x,p)$ is convex, i.e.
	\begin{align}\label{eq:convex}
		p,q\in \mathcal P_{d_Y}^t(Y),\alpha \in [0,1]\Rightarrow \,\,C(x,\alpha p + (1-\alpha) q) \leq \alpha C(x,p) + (1-\alpha) C(x,q),
	\end{align}
		then we say that $C$ satisfies Condition~\textsc{(A+)}.
\end{definition}

We now show that under Condition~\textsc{(A+)} the cost functional defining the weak transport problem is lower semicontinuous:

\begin{proposition}\label{th:gozlanlsc}
	Let $C\colon X\times \mathcal P_{d_Y}^t(Y)\rightarrow \mathbb{R}\cup\{+\infty\}$ satisfy condition~\probref{def:conditionA}. Then the map
	\begin{align}\label{eq:lofty cost map}
	\mathcal P_{\hat d}^{t}(X\times \mathcal P_{d_Y}^t(Y))\ni P \mapsto \int_{X\times \mathcal P_{d_Y}^t(Y)} C(x,p)P( dx,dp)
	\end{align}
	is lower semicontinuous. If $C$ satisfies condition \textsc{(A+)} then the map
	\begin{align}\label{eq:cost map}
	\mathcal P_d^t(X\times Y)\ni \pi \mapsto \int_X C(x,\pi_x)\pi( dx\times Y)
	\end{align}is lower semicontinuous.
\end{proposition}

\begin{proof}
Let $P^k\to P$ in $\mathcal P_{\hat d}^t(X\times \mathcal P_{d_Y}^t(Y))$.
Similar to \cite[Theorem A.3.12]{DuEl11}, we can approximate $C$ from below by $d$-Lipschitz functions and obtain lower semicontinuity of \eqref{eq:lofty cost map}, i.e.,
	$$\liminf_k \int_{X\times \mathcal P(Y)}C(x,p)\,P^k( dx, dp)\geq \int_{X\times \mathcal P(Y)}C(x,p)\,P( dx, dp).$$

To show lower semicontinuity of \eqref{eq:cost map}, let $\pi^k\to \pi$ in $\mathcal P_d^t(X\times Y)$ and denote $P^k=J(\pi^k)$. We may assume that $\liminf_k \int_X C(x,\pi_x^k)\pi^k( dx\times Y) = \lim_k \int_X C(x,\pi_x^k)\pi^k( dx\times Y) $ by selecting a subsequence. By Lemma~\ref{lem:relcompembedding} we know that $\{P^k\}_k$ is relatively compact in $\mathcal P_{\hat d}^{t}(X\times \mathcal P_{d_Y}^t(Y))$. Denote by $P$ an accumulation point of $\{P^k\}_k$. From now on we work along a subsequence converging to $P$. Observe that $$\int_X C(x,\pi_x^k)\pi^k( dx\times Y) = \int_{X\times \mathcal P(Y)}C(x,p)\,P^k(dx, dp).$$
Hence, we find by the first part that
	$$\liminf_k \int_{X\times \mathcal P(Y)}C(x,p)\,P^k( dx, dp)\geq \int_{X\times \mathcal P(Y)}C(x,p)\,P( dx, dp).$$

	Observe that the $X$-marginal of $P$ equals the $X$-marginal of $\pi$, so by convexity of $C(x,\cdot)$ we then have 
	\begin{align*}
	\liminf_k \int_X C(x,\pi_x^k)\pi^k( dx\times Y)&\geq \int_{X\times \mathcal P(Y)}C(x,p)\,P_x( dp)\pi( dx\times Y)\\ &\geq  \int_X C\Big(x,\int_{\mathcal P(Y)}p( dy) P_x( dp)\Big)\pi( dx\times Y).
	\end{align*}
	Now, if $f$ is continuous bounded on $X\times Y$, we have 
	$$\int_{X\times Y} f(x,y)\pi^k( dx, dy)\to \int_{X\times Y} f(x,y){\color{black}\pi}( dx, dy).$$
	 But the function $F(x,p):=\int_Y f(x,y)p( dy)$ is easily seen to be continuous and bounded in $X\times \mathcal P(Y)$. Hence $\int F dP^k \to \int F dP$ and by the structure of $F$ we deduce $$\int_{X\times Y} f(x,y)\pi( dx, dy) = \int F dP = \int_{X\times \mathcal P(Y)}\int_Y f(x,y)p(dy)P( dx, dp). $$
	This shows for the disintegration $(\pi_x)_{x\in X}$ of $\pi$ that $\pi_x( dy)= \int_{\mathcal P(Y)}p( dy)\, P_x( dp)$ for $\pi( dx\times Y)$-almost every $x$. So we conclude
	$$\liminf_k \int_X C(x,\pi_x^k)\pi^k( dx\times Y)\geq \int_X C(x,\pi_x)\pi( dx\times Y).$$
\end{proof}
We are finally ready to provide our main existence result:

\begin{theorem}\label{thm existence compactset}
	Let $C\colon X\times \mathcal P_{d_Y}^t(Y)\rightarrow \mathbb{R}\cup\{+\infty\}$ satisfy Condition~\probref{def:conditionA}. 
	If $\Lambda \subset \mathcal P_{\hat d}^t(X\times \mathcal P_{d_Y}^t(Y))$ is compact, then there exists a minimizer $P^*\in\Lambda$ of
	$$\inf_{P\in\Lambda} \int_{X\times \mathcal P(Y)} C(x,p)P(dx,dp).$$
In particular $\mathcal P (X)\times \mathcal P^t_{d_Y}(Y)\ni(\mu,\nu)\mapsto \hat V(\mu,\nu)$	is lower semicontinuous and  $\hat V(\mu,\nu)$ is attained (recall \eqref{eq:defhatV}). Assume now that $C$ fulfils Condition~\textsc{(A+)} and $\Pi \subseteq \mathcal P_{d}^{t}(X\times Y)$ is compact. Then there exists a minimizer $\pi^*\in\Pi$ of
	$$\inf_{\pi\in\Pi} \int_X C(x,\pi_x)\pi( dx\times Y).$$
	In particular $\mathcal P (X)\times \mathcal P^t_{d_Y}(Y)\ni(\mu,\nu)\mapsto  V(\mu,\nu)$ is lower semicontinuous and $V(\mu,\nu)$ is attained (recall \eqref{eq weak transport def}).

\end{theorem}

\begin{proof}
The existence of minimizers in $\Lambda$ and $\Pi$ are direct consequences of their compactness and the lower semicontinuity of the objective functionals (Proposition \ref{th:gozlanlsc}).

We move to the study of $\hat V$. Let $(\mu_k,\nu_k)\rightarrow (\mu,\nu)$ in $\mathcal P(X) \times (\mathcal P_{d_Y}^t,\mathcal W_t)$. For any $k\in\mathbb{N}$ we find an optimizer $P^*_k$ of $\hat V(\mu_k,\nu_k)$. Note that the set $\{P_k^*\colon k\in\mathbb{N}\}$ is relatively compact in $\mathcal P_{\hat d}^{t}(X\times \mathcal P_{d_Y}^t(Y))$. Therefore, we can find again a converging subsequence with limit point in $\Pi(\mu,\nu)$. Without loss of generality we assume
$$\liminf_k \hat V(\mu_k,\nu_k) = \lim_k \hat V(\mu_k,\nu_k).$$
Using lower semicontinuity of the objective functional shows the assertion for $\hat V$. By Lemma~\ref{lem:convex optimality equivalence} the lower semicontinuity of $V$ is immediate.
\end{proof}

Of course Theorems \ref{thm existence general} and \ref{thm existence marginals} are particular cases of the second half of Theorem \ref{thm existence compactset}. More generally: if $A$ is compact in $\mathcal P(X)$ and $B$ is compact in $(\mathcal P^t_{d_Y}(Y),\mathcal W_t)$, then $\Pi:=\bigcup_{\mu\in A,\nu\in B}\Pi(\mu,\nu)$ is compact in $\mathcal P_{d}^{t}(X\times Y)$ and Theorem \ref{thm existence compactset} applies.

\section{Duality}\label{sec:duality}

We denote by $\Phi_{t}$ the set of continuous functions on $Y$ which satisfy the growth constraint 
$$\exists y_0\in Y,\,\exists a,b\in\R_+,\,\forall y\in Y:\,|\psi(y)|\leq a + bd_Y(y,y_0)^t,$$
and by $\Phi_{b,t}$ the subset of functions in $\Phi_t$ which are bounded from below. Further, we recall the notion of $C$-conjugate :
The $C$-conjugate of a measurable function $\psi \colon Y\rightarrow \R$, denoted $R_C\psi$, is given by
\begin{align}\label{def:RCphi}
R_C\psi(x) := \inf_{p\in \mathcal P_{d_Y}^t(Y)} p(\psi) + C(x,p).
\end{align}
%Observe that alternatively
%$$R_C\psi(x) = \inf_{Q \in \mathcal P_{\mathcal W_t}^t(\mathcal P_{d_Y}^t(Y))} \int_{\mathcal P(Y)} p(\psi) + C(x,p) Q(dp).$$

We obtain Theorem \ref{th:kantorovichduality} as a particular case of the following:

\begin{theorem}\label{thm duality strong}
	Let $C\colon X \times \mathcal P_{d_Y}^t(Y)\rightarrow \mathbb{R}\cup\{+\infty\}$ satisfy Condition~\probref{def:conditionA}. Then
	\begin{align}\label{eq:lofty duality}
		\inf_{P\in\Lambda(\mu,\nu)} \int_{X\times \mathcal P(Y)} C(x,p) P(dx,dp) = \sup_{\psi \in \Phi_{b,t}} -\nu(\psi) + \int_X R_C\psi(x)\mu(dx).
	\end{align}
If moreover $C$ satisfies Condition~\textsc{(A+)}, then
	\begin{align}\label{eq:duality}
	V(\mu,\nu):=\inf_{\pi\in\Pi(\mu,\nu)} \int_X C(x,\pi_x) \mu(dx) = \sup_{\psi \in \Phi_{b,t}} -\nu(\psi) + \int_X R_C\psi(x)\mu(dx).
	\end{align}
\end{theorem}

\begin{remark}
A proof of Theorem \ref{th:kantorovichduality} can be obtained by means of \cite[Theorem 9.6]{GoRoSaTe17}, since we may verify the hypotheses therein thanks to our Proposition \ref{th:gozlanlsc}. We prefer to obtain the slightly stronger Theorem \ref{thm duality strong} via self-contained arguments. The primal-dual equality \eqref{eq:duality} was obtained in \cite[Theorem 4.2]{AlBoCh18} in the case when $X,Y$ are compact spaces.
\end{remark}

\begin{proof}[Proof of Theorem \ref{thm duality strong}]
Fix $y_0\in Y$. Define the auxiliary cost function $\widetilde C\colon X\times \mathcal P_{d_Y}^t(Y)$ by
$$\widetilde C(x,p) := C(x,p) + \mathcal W_t(p,\delta_{y_0})^t$$
and $F\colon \mathcal P_{d_Y}^t(Y)\rightarrow \R\cup\{+\infty\}$ by
	\begin{align}F(m) &:= \inf_{P \in \Lambda(\mu,m)} \int_{X\times \mathcal P(Y)}\widetilde C(x,p)P(dx,dp)\notag\\
	&= \inf_{P \in \Lambda(\mu,m)} \int_{X\times \mathcal P(Y)} C(x,p)  P(dx,dp) + \int_Y d_Y(y,y_0)^tm(dy).\label{eq F C d}
	\end{align}
	Since the integrand $\widetilde C$ is bounded from below and lower semicontinuous we can apply Proposition~\ref{th:gozlanlsc} and find that $F$ is lower semicontinuous on $\mathcal P_{d_Y}^t(Y)$.
	Note that for any $\alpha \in [0,1]$ and $m_1,m_2\in\mathcal P_{d_Y}^t(Y)$ we have
	$$P_i \in \Lambda(\mu,m_i),~i=1,2~\implies \alpha P_1 + (1-\alpha)P_2 \in \Lambda(\mu,\alpha m_1 + (1-\alpha) m_2),$$
	and, particularly, it follows that $F$ is convex. We can extend $F$ to the set $\mathcal M_{d_Y}^t(Y)$ of bounded signed measures  with finited $t$-moment (i.e.\ $m\in \mathcal M_{d_Y}^t(Y)$  implies $\int_Y d_Y(y,y_0)^t|m|(dy)<\infty$ for some $y_0$) by setting $F(m) = +\infty$ if $m\notin \mathcal P_{d_Y}^t(Y)$. We equip the space $\mathcal M_{d_Y}^t(Y)$ with the topology {\color{black}induced} by $\Phi_t$. It follows that the extension of $F$ is still convex and lower semicontinuous. Now, the spaces $\Phi_t$ and $\mathcal M_{d_Y}^t(Y)$ are in separating duality. Define the convex conjugate $F^*\colon \Phi_t \rightarrow \R\cup\{+\infty\}$ of $F$ by
	\begin{align}\label{eq:defF*}
	F^*(\psi) = \sup_{m\in\mathcal P_{d_Y}^t(Y)} m(\psi) - F(m).	
	\end{align}		
Observe that  $F^*(\psi)=\lim_{k\to +\infty} F^*(\psi\wedge k)$, by monotone convergence. 
	We may apply the Fenchel duality theorem~\cite[Theorem 2.3.3]{Za02}, and then replace $\Phi_t$ by $\Phi_{b,t}$, obtaining:
	\begin{align*}
		F(m) &= \sup_{\psi \in \Phi_t} m(\psi) - F^*(\psi) \\
		&= \sup_{ -\psi \in \Phi_{b,t}} m(\psi) - F^*(\psi) \\
		&= \sup_{\psi \in \Phi_{b,t}} m(-\psi) - F^*(-\psi).
	\end{align*}
	Now we show that
	\begin{align}\label{eq:FstartildeRC}
	F^*(-\psi) = -\int_X  R_{\widetilde C}\psi(x)\mu(dx).
	\end{align}		
	Rewriting \eqref{eq:defF*} yields
	\begin{align*}
		F^*(-\psi) &= \sup_{m\in\mathcal P^t_{d_Y}(Y)} m(-\psi) - \inf_{P\in\Lambda(\mu,m)} \int_{X\times \mathcal P(Y)} \widetilde C(x,p)P(dx,dp) \\
		&= \sup_{\substack{m\in\mathcal P_{d_Y}^t(Y)\\ P\in \Lambda(\mu,m)}}
		-\int_X \left(  \int_{\mathcal P(Y)}  p(\psi) + \widetilde C(x,p)P_x(dp) \right) \mu(dx) \\
		&= -\inf_{\substack{m\in\mathcal P_{d_Y}^t(Y)\\ P\in \Lambda(\mu,m)}}
		\int_X \left( \int_{\mathcal P(Y)} p(\psi) + \widetilde C(x,p) P_x(dp)\right)\mu(dx)\\
		&\leq -\int_X R_{\widetilde C}\psi(x)\mu(dx).
	\end{align*}
	To show the converse inequality, we assume without loss of generality that 
	$$\int_X \widetilde{R_C}\psi(x)\mu(dx)<+\infty.$$ For all $x\in X$ the value of $R_{\widetilde C}\psi(x)$ is finite, because $\psi$ is bounded from below. Fix $\epsilon>0$. The map $R_{\widetilde C}\psi(\cdot)$ is lower semianalytic by \cite[Proposition 7.47]{BeSh78} and by \cite[Proposition 7.50]{BeSh78} there exists
	%Note that the maps $\widetilde{R_C}\psi$ and
%	$$(x,Q) \mapsto \int_{\mathcal P(Y)} p(\psi+f) + C(x,p)Q(dp),$$
%	are Borel-measurable on $X$ and $X\times \mathcal P_{\mathcal W_t}^t(\mathcal P_{d_Y}^t(Y))$, respectively. Thus, the set
%	$$A = \left\{(x,Q) \in X \times \mathcal P_{\mathcal W_t}^t(\mathcal P_{d_Y}^t(Y))\colon \widetilde{R_C}\psi(x) + \epsilon \geq \int_{\mathcal P(Y)} p(\psi+f) + C(x,p)Q(dp) \right\}$$
%	is Borel-measurable. By the Jankov-von Neumann uniformization theorem \cite[Theorem 18.1]{Ke95} 	
%	  we find 
	  an analytically measurable probability kernel $(\tilde p_x)_{x\in X}\in (\mathcal P_{d_Y}^t(Y))^X$ such that for all $x\in X$
	$$p_x(\psi) + \widetilde C(x,p_x)\leq R_{\widetilde C}\psi(x) + \epsilon.$$
	Then, we immediately obtain
	$$\int_X p_x(\psi) + \widetilde C(x,p_x) \mu(dx)\leq \int_X R_{\widetilde C}\psi(x)\mu(dx) + \epsilon.$$
	%By \cite[Proposition 7.45]{BeSh04}
	The term $\delta_{p_x}(dp)\mu(dx)$ uniquely defines a probability {\color{black}measure} $\tilde P$ on $X\times \mathcal P(Y)$.
	%$$\tilde P(A\times B) = \int_A \tilde P_x(B) \mu(dx)\quad \forall A \in \mathcal B(X),~B\in\mathcal B(\mathcal P(Y)).$$	
	Since $\widetilde C$ and $\psi$ are bounded from below, we infer that 
	$$\mathcal W_t(\proj_Y\hat I(\tilde P),\delta_{y_0})^t = \int_{X\times \mathcal P(Y)} \mathcal W_t(p,\delta_{y_0})^t \tilde P(dx,dp)< +\infty,$$
	 and in particular $\proj_Y \hat I(\tilde P)\in \mathcal P_{d_Y}^t(Y)$. Clearly $\tilde P\in\Lambda(\mu,\proj_Y\hat I(\tilde P))$, so
	\begin{align*}
	-\int_X \widetilde{R_C}\psi(x)\mu(dx) &\leq \proj_Y(\hat I(\tilde P))(-\psi) - \int_{X\times \mathcal P(Y)} C(x,p) + \mathcal W_t(p,\delta_{y_0})^t\tilde P(dx,dp) +\epsilon \\
	&\leq \proj_Y(\hat I(\tilde P))(-\psi) - F\left( \proj_Y \hat I(\tilde P) \right ) +\epsilon  \\
	&\leq F^*(-\psi) + \epsilon,
	\end{align*}
	and since $\epsilon$ was arbitrary, we have shown \eqref{eq:FstartildeRC}.
	
	So far, we know that
	$$F(m) = \sup_{\psi\in\Phi_{b,t}} -m(\psi) + \int_X R_{\widetilde C}\psi(x)\mu(dx).$$
	Define $f(y): = d_Y(y,y_0)^t$ and note that
	$R_C(\psi+f)(x) = R_{\widetilde C}\psi(x)$ for all $x\in X$, as well as $\psi + f \in \Phi_{b,t}$ for $\psi \in \Phi_{b,t}$. From \eqref{eq F C d} we get
	\begin{align*}
	\inf_{P\in\Lambda(\mu,m)} P(C) &= F(m)- \mathcal W_t(m,\delta_{y_0})^t \\
	&= \sup_{\psi\in\Phi_{b,t}} -m(\psi + f) + \int_X R_{\widetilde C}\psi(x)\mu(dx)\\
	&= \sup_{\psi\in\Phi_{b,t}} -m(\psi) + \int_X R_C \psi(x)\mu(dx),
	\end{align*}	
	which shows \eqref{eq:lofty duality}.
		
	If for all $x\in X$ the map $C(x,\cdot)$ is convex, 
	%we find {\color{red}$\widehat{R_C}\psi(x) = R_C\psi(x)$}, since $\widehat{R_C}\psi(x) \leq R_C\psi(x)$ 
%	and for any $Q\in \mathcal P_{\mathcal W_t}^t(\mathcal P_{d_Y}^t(Y))$ we have
%	$$\int_{\mathcal P(Y)} p(\psi) + C(x,p)Q(dp) \geq I(Q)(\psi) + C(x,I(Q)) \geq R_C\psi(x).$$
%	
%	
	then \eqref{eq:duality} follows by Lemma~\ref{lem:convex optimality equivalence} and \eqref{eq:lofty duality}.
\end{proof}

\section{On the restriction property}

The restriction property of optimal transport roughly states that if a coupling is optimal, then the conditioning of the coupling to a subset is also optimal given its marginals. This property fails for weak optimal transport, as we illustrate with {\color{black}a simple} example:

\begin{example}
Let $X=Y=\mathbb R$, $\mu=\frac{1}{2} \delta_{-1}+\frac{1}{2}\delta_1$, $\nu=\frac{1}{4}\delta_{-2}+\frac{1}{2}\delta_0+\frac{1}{4}\delta_2$ and $C(x,\rho)=\left(x-\int y\rho(dy)\right)^2$. We consider the weak transport problem with these ingredients, and observe that an optimal coupling is given by
$$\pi = \frac{1}{4}[\delta_{(1,2)}+\delta_{(1,0)}+\delta_{(-1,0)}+\delta_{(-1,-2)}],$$
since it produces a cost equal to zero. Consider the set $K=\{(x,y):y\neq 0\}$ and $\tilde\pi(dx,dy) = \pi(dx,dy|K)$ the conditioning of $\pi$ to the set $K$, i.e. $\tilde \pi(S):=\frac{\pi(S\cap K)}{\pi(K)}$. It follows that $$\tilde\pi= \frac{1}{2}[\delta_{(1,2)}+\delta_{(1,-2)}],$$
and denoting by $\tilde \mu $ and $\tilde \nu$ the first and second marginals of $\tilde \pi$, we have $\tilde\mu=\mu$ and $\tilde \nu =\frac{1}{2}\delta_2+\frac{1}{2}\delta_{-2}$. With $\tilde \mu$ and $\tilde \nu$ and again the cost $C$ as ingredients, an optimizer for the weak transport problem is given by 
$$\hat\pi =\frac{3}{8}\delta_{(1,2)}+\frac{1}{8}\delta_{(1,-2)}+\frac{1}{8}\delta_{(-1,2)}+\frac{3}{8}\delta_{(-1,-2)},$$
since this time this coupling produces a cost equal to zero. On the other hand the cost of $\tilde \pi$ is equal to $1$, and so $\tilde \pi$ is not optimal between is marginals.
\end{example}

However, we can state the following positive result.\footnote{In a preliminary version of this article the restriction property Proposition \ref{prop restriction} was used to derive Theorem \ref{thm a la GJ} from the compact version given by Gozlan and Juillet \cite{GoJu18}. Following the insightful  suggestion of the  anonymous referee, we now give a more self contained argument that does not require Proposition \ref{prop restriction}  /  \cite{GoJu18}. We have decided to keep Proposition \ref{prop restriction} since it might be of some independent interest.}
\begin{proposition}\label{prop restriction}
Suppose that $\pi$ is optimal between the marginals $\mu$ and $\nu$, $V(\mu,\nu)<\infty$, and that $C(x,\cdot)$ is convex. Let $0 \leq\tilde \mu\leq \mu$ be a non-negative measure such that $0\not\equiv\tilde\mu $ and define $\hat\mu = \tilde\mu/\tilde\mu(X)$. Then $\hat\pi(dx,dy):=\hat\mu(dx) \pi_x(dy)$ is optimal between its marginals. 
\end{proposition}

\begin{proof}
By contradiction, suppose there exists a coupling $\chi$ with the same marginals as $\hat\pi$ such that 
$$\int C(x,\chi_x)\hat\mu(dx)<\int C(x,\hat\pi_x)\hat\mu(dx).$$
Now define $\pi^*:=\pi + \tilde{\mu}(X)[\chi-\hat \pi]=\pi-\tilde\mu.\pi_x+\tilde{\mu}(X) \chi$. Observe that $\pi^*$ has marginals $\mu,\nu$, and $\pi^*(X\times Y)=1$. We also have $\pi^*\geq 0$ since $\tilde\mu \leq \mu$, so $\pi^*$ is a probability measure. Of course $0\leq\frac{d\tilde\mu}{d\mu}\leq 1$ and clearly $\pi^*_x=\left(1-\frac{d\tilde\mu}{d\mu}(x) \right)\pi_x+ \frac{d\tilde\mu}{d\mu}(x)\chi_x$. Therefore
\begin{align*}
\int C(x,\pi^*_x)\mu(dx) &= \int C\left(x,\left(1-\frac{d\tilde\mu}{d\mu}(x) \right)\pi_x+ \frac{d\tilde\mu}{d\mu}(x)\chi_x\right )\mu(dx)\\
&\leq \int C(x,\pi_x)\mu(dx) + \int [C(x,\chi_x)-C(x,\pi_x)]  \tilde\mu(dx)\\
&< \int C(x,\pi_x)\mu (dx),
\end{align*}
where we used convexity in the first inequality and that $V(\mu,\nu)<\infty$ in the second one. 
\end{proof}

\section{$C$-Monotonicity for weak transport costs}

Cyclical monotonicity plays a crucial role in classical optimal transport \cite{RaRu90, GaMc96}. This has inspired similar development for weak transport costs in \cite{BaBeHuKa17,GoJu18}:

\begin{definition}[$C$-monotonicity]\label{def:C-monotonicity}
	We say that a coupling $\pi \in \Pi(\mu,\nu)$ is $C$-monotone if there exists a measurable set $\Gamma\subseteq X$ with $\mu(\Gamma)=1$, such that for any finite number of points $x_1,\dots,x_N$ in $\Gamma$ and measures $m_1,\dots,m_N$ in $\mathcal P(Y)$
	with $\sum_{i=1}^N m_i = \sum_{i=1}^N \pi_{x_i}$, the following inequality holds:
	\begin{align*}
		\sum_{i=1}^N C(x_i,\pi_{x_i}) \leq \sum_{i=1}^N C(x_i,m_i).
	\end{align*}
\end{definition}

We first show that $C$-monotonicity is necessary for optimality under minimal assumptions. We then provide strengthened assumptions under which $C$-monotonicity is sufficient.

\subsection{$C$-monotonicity: necessity}
 
We denote by $S_N$ the set of permutations of the set $\{1,\dots,N\}$. If $\vec z:=(z_1\dots,z_n)$ is any $N$-vector, and $\sigma\in S_N$, we naturally overload the notation by defining $$\sigma(\vec z):=(z_{\sigma(1)},\dots,z_{\sigma(N)}).$$ 

Recall the notation \eqref{eq weak transport def} for the weak transport problem{\color{black}, and the following lemma, which is employed prominently in the proof of Theorem~\ref{thm cyclical}.
\begin{lemma}[{\cite[Proposition 2.1]{BeGoMaSc08}}]\label{prop:kellerers lemma}
	Let $X_1,\ldots, X_n$, $n\geq 2$, be Polish spaces equipped with probability measures $\mu_i \in \mathcal P(X_i)$, $i=1,\ldots,n$. Then for any analytic set $B\subset X_1\times \cdots \times X_n$ one of the following holds:
	\begin{enumerate}[(a)]
		\item For every $i = 1,\ldots, n$ there is a $\mu_i$-null set $A_i \subset X_i$ s.t.
		$$B \subset \bigcup_{i=1}^n \proj_{X_i}^{-1}(A_i).$$
		\item There exists a coupling $\pi \in \Pi(\mu_1,\ldots,\mu_n)$ with $\pi(B) > 0$.
	\end{enumerate}
\end{lemma}
The previous lemma is originally stated only for Borel sets, but the same proof technique also works for analytic sets.
}

 Our main result, concerning the necessity of $C$-monotonicity is the following:

\begin{theorem}\label{thm cyclical}
	Let $C$ be jointly measurable and $C(x,\cdot)$ be convex and lower semicontinuous for all $x$. Assume that $\pi^*$ is optimal for $V(\mu,\nu)$ and $|V(\mu,\nu)|<\infty$. Then $\pi^*$ is $C$-monotone.
\end{theorem}

\begin{proof}
	Let $N\in \N$. Then
	\begin{align*}
	\mathcal D_N:=\bigg\{((x_1,\ldots,x_N)&,(m_1,\ldots,m_N))\in X^N\times\mathcal P(Y)^N:\\ &\sum_{i=1}^N \pi^*_{x_i} = \sum_{i=1}^N m_i \text{ and } \sum_{i=1}^N C(x_i,\pi^*_{x_i}) >\sum_{i=1}^N C(x_i,m_i)\bigg\},
	\end{align*}
	is an analytic set. Write $$D_N:= \proj_{X^N}(\mathcal D_N).$$ By Jankov-von Neumann uniformization \cite[Theorem 18.1]{Ke95} there is an analytically measurable function $f_N\colon D_N\rightarrow \mathcal P(Y)^N$ such that $\text{graph}(f_N) \subset \mathcal D_N$. We can extend $f_N$ to $X^N$ by defining it on $X^N \setminus D_N$ as the Borel-measurable map $\vec x \mapsto (\pi^*_{x_1},\ldots,\pi^*_{x_N})$.
Observe that for all $\sigma \in S_N$, we have $(\sigma,\sigma)(\mathcal D_N) = \mathcal D_N$. Thanks to this, and Lemma~\ref{lem:extendf} below, we can assume without loss of generality that $f_N$ satisfies
	$$f_N\circ \sigma = \sigma \circ f_N\quad \forall \sigma\in S_N.$$
	We write $f^i_N(\vec x)$ for the $i$-th element of the vector $f_N(\vec x)\in\mathcal P(Y)^N$.

Assume that there exists a coupling $Q\in\Pi(\mu^N)=\Pi(\mu,\ldots,\mu)$ such that $Q(D_N) > 0$. We now show that this is in conflict with optimality of $\pi^*$. We clearly may assume that $Q$ is symmetric, i.e.\ such that for all $\sigma \in S_N$ we have $Q(B)=Q(\sigma(B))$ for all $B\in\mathcal B(X^N)$ (in other words $\sigma(Q) = Q$). We define the possible contender $\tilde \pi$ of $\pi^*$ by
	\begin{align}
	\tilde \pi(dx_1,dy) := \mu(dx_1) \int_{X^{N-1}}  Q_{x_1} ( dx_2,\ldots, dx_n) f_N^1(x_1,\ldots,x_N)(dy),
	\end{align}
	which is legitimate owing to all measurability precautions we have taken. We will prove 
	\begin{enumerate}
		\item $\tilde\pi\in\Pi(\mu,\nu)$,
		\item $\int\mu( dx)C(x,\pi^*_x) > \int \mu( dx)C(x, \tilde\pi_x )$.
	\end{enumerate}
	
	Ad (1): Evidently the first marginal of $\tilde\pi$ is $\mu$. Write $\sigma_i\in S_N$ for the permutation that merely interchanges the first and $i$-th component of a vector. By the symmetric properties of $Q$ and $f_N$ we find
	\begin{align*}
	\int_{X} \mu( dx_1)\tilde\pi_{x_1}(dy)&=\int_{X^N}  Q(dx_1,\ldots,dx_N)f^1_N(\vec x)(d y)\\
	&=\frac{1}{N}\sum_{i=1}^N \int_{X^N} \sigma_i(Q)(dx_1,\ldots,dx_N) f^i_N(\vec x)(dy)\\
	&=\frac{1}{N}\sum_{i=1}^N \int_{X^N} Q(dx_1,\ldots,dx_N) \pi^*_{x_i}(dy)\\
	&=\nu(dy).
	\end{align*}

	Ad (2): On $D_N$ holds by construction the strict inequality $$\sum_{i=1}^N C(x_i,f_N^i(\vec x))<\sum_{i=1}^N C(x_i,\pi_{x_i}).$$
	 Using convexity of $C(x,\cdot)$ and the symmetry properties of $Q$ and $f_N$, we find
	\begin{align*}
	\int_XC(x,\tilde\pi_{x})\mu(dx)&=\int_X\mu( d{x_1})C\left({x_1},\int_{X^{N-1}}Q_{x_1}( dx_2,\ldots,dx_N)f_N^1(\vec x)\right)
	\\ &\leq \int_{X^N}Q(d\vec x)C\big({x_1},f^1_N(\vec x)\big)\\
	&=\frac{1}{N}\sum_{i=1}^N \int_{X^N} Q(d\vec x)C\big({x_i},f^i_N(\vec x)\big)\\
	&<\frac{1}{N}\sum_{i=1}^N \int_{X^N} Q(d\vec x)C\big({x_i},\pi_{x_i}\big) = \int_X C(x,\pi_x)\mu(dx),
	\end{align*}
	yielding a contradiction to the optimality of $\pi^*$.
	
	We conclude that no measure $Q$ with the stated properties exists. {\color{black}By Lemma~\ref{prop:kellerers lemma}}, we obtain that $D_N$ is contained in a set of the form $\bigcup_{k=1}^N \proj_k^{-1}(M_N)$ where $\mu(M_N) = 0$ and $\proj_k$ denotes the projection from $X^N$ to its $k$-th component. Since $N\in\mathbb N$ was arbitrary, we can define the set $\Gamma:= (\bigcup_{N\in\mathbb{N}} M_N)^C$ with $\mu(\Gamma) = 1$, which has the desired property.
\end{proof}

The missing bit in the above proof is Lemma \ref{lem:extendf}. By \cite[Theorem 7.9]{Ke95} there exists for every Polish space $X$ a closed subset $F$ of the Baire space $\mathcal N := \N^\N$ and a continuous bijection $h_X\colon F \rightarrow X$. On the Baire space the lexicographic order naturally provides a total order. Hence, $X$ inherits the total order of $F\subset \mathcal N$ by virtue of $h_X$ and its Borel-measurable inverse $h_X^{-1} := g_X$, namely:
$$x,y \in X\colon x\leq y \Leftrightarrow h_X^{-1}(x) = g_X(x) \leq h_X^{-1}(y) = g_X(y).$$

\begin{lemma}\label{lem:extendf}
	The set 
	$$A = \left\{\vec x\in X^N\colon x_1\leq x_2\leq \ldots\leq x_N \right\},$$
	is Borel-measurable. Given $f\colon A\subset X^N\rightarrow Y^N$ an analytically measurable function, there exists an analytically measurable extension $\hat f\colon X^N\rightarrow Y^N$ such that for any $\sigma \in S_N$
	$$\hat f \circ \sigma = \sigma \circ \hat f.$$
\end{lemma}

\begin{proof}[Proof of Lemma~\ref{lem:extendf}]Let $\hat A = \left\{\vec a\in \mathcal N^N\colon a_1\leq a_2\leq \ldots\leq a_N \right\}$, and define $g\colon \mathcal N^N \rightarrow S_N$ by $ g(\vec a) = \sigma$ where $\sigma\in S_N$ satisfies
	\begin{itemize}
	\item $\sigma(\vec a) \in \hat A$
	\item for each $i,j$ such that $ 0\leq i<j \leq N$ it holds $$ a_i = a_j \implies \sigma(i) < \sigma(j).$$
	\end{itemize}
	With these precautions $ g(\vec a) = \sigma$ is indeed well defined.
	For each $\sigma \in S_N$ we define also $B_\sigma\subset \mathcal N^N$ by
	$$B_\sigma :=\left\{ \vec a \in \mathcal N^N\colon g(\vec a) = \sigma\right\}= \left\{\vec a \in \mathcal N^N\colon  a_{\sigma(1)} \leq^1_\sigma a_{\sigma(2)} \leq^2_\sigma \ldots \leq^{N-1}_\sigma a_{\sigma(N)}\right\},$$
	where the order $\leq_\sigma^i$ is defined depending on $\sigma$ by
	$$\leq^{i}_\sigma := \begin{cases} \leq & \sigma(i) \leq \sigma(i+1),\\ < & \text{else}.\end{cases}$$
	It follows from this representation that $B_\sigma $ is Borel-measurable. We introduce
	$$X^N\ni\vec x \mapsto g^N_X(\vec x):= (g_X( x_1), g_X( x_2),\dots,g_X( x_N))\in F^N\subset \mathcal N^N.$$ 
	Then the set $$A_\sigma := \{\vec x \in X^N\colon g\circ g^N_X(\vec x) = \sigma\} = (g_X^N)^{-1}(B_\sigma),$$
	is Borel-measurable. In particular, $A_{id} = A$ is Borel-measurable. Note that $\cup_\sigma A_\sigma = X^N$ and $A_{\sigma_1} \cap A_{\sigma_2} = \emptyset$ if $\sigma_1\not\equiv \sigma_2$. 
	We can apply Lemma~\ref{lem:lex<perm}, proving the continuity\footnote{In fact one obtains $	\max_{i\in \{1,\ldots,N\}} d_{\mathcal N}(g(\vec a)(\vec a)_i,g(\vec b)(\vec b)_i) \leq \max_{i \in \{1,\ldots,N\}} d_{\mathcal N}(a_i,b_i)$, for $d_{\mathcal N}$ the metric on ${\mathcal N}$ that we recall in Lemma~\ref{lem:lex<perm}.} of $$\mathcal N^N\ni \vec a\mapsto G(a):=g(\vec a)(\vec a) \in \mathcal N^N.$$ We define the candidate for the desired extension of $f$ by
	\begin{align*}
		\hat f\colon& X^N\rightarrow Y^N,\\
		&\,\,\vec x\,\, \mapsto  (g\circ g_X^N(\vec x))^{-1} \left (\,f \circ (g_X^N)^{-1}\circ G \circ g_X^N (\vec x)\,\right ), 
	\end{align*}
	which is well defined since $G \circ g_X^N (\vec x)\in \hat A$, so that $(g_X^N)^{-1}\circ G \circ g_X^N (\vec x)\in A$. As a composition of analytically measurable function, $\hat f$ inherits this property. It is also clear that $\hat f (\vec x)=f(\vec x)$ if $\vec x\in A$. Finally, for any $\sigma \in S_N$ and $\vec x \in X^N$, we easily find
	\begin{align*}
	\sigma^{-1}(\hat f\circ \sigma(\vec x) )&=  \hat f(\vec x).
	\end{align*}		
	\end{proof}

\begin{lemma}\label{lem:lex<perm}
	Let each of $a,b\in \mathcal N^N$ be increasing vectors.\footnote{A vector $v = (v_i)_{i=1}^N\in\mathcal N^N$ is increasing if for any $1 \leq i < j \leq N$ we have $v_i \leq v_j$, where inequality here is meant in the lexicographic order on $\mathcal{N}$.} Then for any permutation $\sigma \in S_N$ we have
	\begin{align}\label{eq:lex<perm}
	\max_{i\in \{1,\ldots,N\}} d_{\mathcal N}(a_i,b_i) \leq \max_{i \in \{1,\ldots,N\}} d_{\mathcal N}(a_i,b_{\sigma(i)}),
	\end{align}
	where the metric $d_{\mathcal N}$ on $\mathcal N$ is given by
	$$d_{\mathcal N}(a,b) = \begin{cases} 0 & a = b\\ \frac{1}{\min\{n\in\N\colon a(n) \neq b(n)\}} & \text{else}.\end{cases}$$
\end{lemma}

\begin{proof}
	We show the assertion by induction.
	For $N=1$ \eqref{eq:lex<perm} holds trivially.
	Now assume that \eqref{eq:lex<perm} holds for $N=k$.
	Given $\sigma \in S_{k+1}$ and $a,b\in \mathcal N^{k+1}$ increasing, we know  that any $\tilde \sigma \in S_k$ yields
	$$\max_{i\in\{1,\dots,k\}}d_{\mathcal N}(a_i,b_i) \leq \max_{i \in \{1,\dots,k\}} d_{\mathcal N}(a_i, b_{\tilde \sigma(i)}).$$
	If $\sigma(k+1) = k+1$ the assertion follows by the inductive hypothesis. So let $\sigma(k+1) \neq k+1$ and write $k_1 = \sigma(k+1)$ and $k_2 = \sigma^{-1}(k+1)$. Define a permutation $\hat \sigma \in S_k$ by 
	$$\hat \sigma(i) = \begin{cases} \sigma(i) & i\neq k_1 \\
	k_2& i = k_1\end{cases}$$
		Since that $a_{k_2} \leq a_{k+1}$ and $b_{k_1} \leq b_{k+1}$, then 	\begin{align*}
	 a_{k_2} \leq b_{k_1} \implies  a_{k_2} \leq b_{k_1}\leq b_{k+1} \implies  d_{\mathcal N}(a_{k_2},b_{k_1}) \leq d_\mathcal N(a_{k_2},b_{k+1}),\\
	 a_{k_2} \geq b_{k_1} \implies a_{k+1}\geq a_{k_2} \geq b_{k_1} \implies d_{\mathcal N}(a_{k_2},b_{k_1}) \leq d_\mathcal N(a_{k+1},b_{k_1}),
	\end{align*}
	and particularly
	\begin{align}\label{eq:lex<perm1}
	\max_{i\in\{1,\ldots,k\}} d_\mathcal N(a_i, b_{\hat \sigma(i)}) \leq \max_{i\in\{1,\dots,k+1\}} d_\mathcal N(a_i, b_{\sigma(i)}).
\end{align}		
	On the other hand, clearly
		\begin{align*}
	a_{k+1} \geq b_{k+1} \implies d_\mathcal N(a_{k+1},b_{k+1}) \leq d_\mathcal N(a_{k+1},b_{k_1}),\\
	a_{k+1} \leq b_{k+1} \implies d_\mathcal N(a_{k+1},b_{k+1}) \leq d_\mathcal N(a_{k_2},b_{k+1}).
	\end{align*}
	This and \eqref{eq:lex<perm1} yield $\max_{i\in\{1,\dots,k+1\}} d_\mathcal N(a_i,b_i) \leq \max_{i\in\{1,\dots,k+1\}} d_\mathcal N(a_i,b_{\sigma(i)})$, so concluding the inductive step.
\end{proof}

\subsection{C-monotonicity: sufficiency}

The conditions under which Theorem~\ref{thm cyclical} holds are rather mild. If we assume
further continuity properties of $C$, the next theorem establishes that $C$-monotonicity
is also a sufficient criterion for optimality, resembling the classical case. For weak transport costs, we don't know of any comparable result in the literature.

We recall that, for the given compatible complete metric $d_Y$ on $Y$, we denote by $\mathcal W_1$ the $1$-Wasserstein distance \cite[Chapter 7]{Vi03}.  

\begin{theorem}\label{th:compactsufficient}
	Let $\nu\in\mathcal{P}_{d_Y}^1(Y)$. Assume that $C\colon X\times\mathcal{P}_{d_Y}^1(Y)\to\mathbb{R}$ satisfies Condition~\textsc{(A+)} and is $\mathcal W_1$-Lipschitz in the second argument is the sense that
	{\color{black}for some $L\geq 0$:}
	\begin{align}\label{def:Lip W 1}
	|C(x,p)-C(x,q)|\leq L\mathcal W_1(p,q),\,\, \forall x\in X,\forall p,q\in \mathcal{P}_{d_Y}^1(Y).
	\end{align} If $\pi$ is $C$-monotone %and $\int_X C(x,\pi_x)\mu(dx)$ is finite, 
then $\pi$ is an optimizer of $V(\mu,\nu)$.
\end{theorem}

In the proof we will use the following auxiliary result, which we will establish subsequently:

\begin{lemma}\label{lem:lipschitzduality}
	Let $\nu\in\mathcal P_{d_Y}^1(Y)$. Assume that $C\colon X\times \mathcal P_{d_Y}^1(Y)\rightarrow \mathbb{R}$ satisfies Condition~\textsc{(A+)} and is $\mathcal W_1$-Lipschitz in the sense of \eqref{def:Lip W 1}. Then
	\begin{align} \label{eq:lipschitzonly}
		\inf_{\pi \in \Pi(\mu,\nu)}\int C(x,\pi_x)\,\mu( dx) &=
		\sup_{\substack{\phi\in \Phi_{b,1}\\ \lVert \phi\rVert_{Lip}\leq L}} 
		\mu(R_C \phi) -\nu(\phi),
	\end{align}
	where $R_C\phi$ is defined as in~\eqref{def:RCphi}.
\end{lemma}
 
\begin{proof}[Proof of Theorem~\ref{th:compactsufficient}]
Let $\pi$ be $C$-monotone. There is an increasing sequence $(K_n)_{n\in\N}$ of compact sets on $Y$ such that $\nu(K_n)\nearrow 1$. From this we can refine the $\mu$-full measurable set $\Gamma$ in the definition of $C$-monotonicity, see Definition~\ref{def:C-monotonicity}, so that for each $x\in\Gamma$ we have $\lim_n \pi_x(K_n) = 1$ and $\pi_x \in \mathcal P^1_{d_Y}(Y)$. Our goal is to construct a dual optimizer $\phi \in \Phi_{1}$ to $\pi$ such that $$\pi_x(\phi) + C(x,\pi_x) - R_C\phi(x) = 0\quad \forall x\in \Gamma.$$
When this is achieved, Theorem \ref{th:kantorovichduality} and the following arguments show that $\pi$ is optimal as desired:
	\begin{align*}
		\int_X C(x,\pi_x)\mu({d}x)&= \int_\Gamma C(x,\pi_x)\mu({d}x)
		= \int_\Gamma [R_C(\phi)(x) -\pi_x(\phi)]\mu({d}x) \\
		&\leq \liminf_{k\rightarrow -\infty} \int_X [R_C(\phi\vee k)(x) - \pi_x(\phi\vee k)]\mu(dx)\\ &\leq \sup_{\phi\in \Phi_{b,1}} \mu( R_C\phi) - \nu(\phi)
		\\& \leq \inf_{\tilde{\pi}\in\Pi(\mu,\nu)} \int_X C(x,\tilde{\pi}_x)\mu({d}x),
	\end{align*}
where we used that 
$$\liminf_{k\rightarrow -\infty} R_C(\phi\vee k)(x) = \inf_{k \leq 0} R_C(\phi\vee k)(x) = R_C\phi(x)\quad \forall x\in X.$$

Let us prove the existence of a dual optimizer in $\Phi_1$. Let $G\subseteq \Gamma$ be a finite subset. By definition of $C$-monotonicity, we conclude that the coupling $\frac{1}{|G|}\sum_{x_i\in G}\delta_{x_i}(dx)\pi_{x_i}(dy)$ is optimal for the weak transport problem determined by the cost $C$ and its first and second marginals. We can apply Lemma~\ref{lem:lipschitzduality} in this context and obtain
	\begin{gather}\label{eq:equivalence1}
		\inf\limits_{\lVert \phi\rVert_{Lip}\leq L} \sum_{x\in G} \pi_x(\phi) + C(x,\pi_x) - R_C\phi(x) = 0.
	\end{gather}
	We fix $y_0\in K_1 $ and, without loss of generality, find a maximizing sequence $(\phi_k)_{k\in\N}$ of \eqref{eq:equivalence1} such that for all $k\in\N$ the function $\phi_k$ is $L$-Lipschitz and $\phi_k(y_0) = 0$. Note that for all $x\in G$
	$$\pi_x(\phi_k) + C(x,\pi_x) - R_C\phi_k(x) \rightarrow 0,$$
	since by definition $\pi_x(\phi_k) + C(x,\pi_x) - R_C\phi_k(x) \geq 0$. By the Arzel\`a-Ascoli theorem we find for any $n\in\N$ a subsequence of $(\phi_k)_{k\in \N}$ and a $L$-Lipschitz continuous function $\psi_n$ on $K_n$ such that
	$$\lim_j \phi_{k_j}(y) = \psi_n(y)\quad \forall y \in K_n.$$
	Thus by a diagonalization argument we can assume without loss of generality that the maximizing sequence converges uniformly for every $K_n$ to a given $L$-Lipschitz function $\tilde \psi$ defined on  $$A := \bigcup_n K_n. $$  We can extend $\tilde \psi$ from $A$ to all of $Y$, obtaining an everywhere $L$-Lipschitz function, via
	\begin{align}\label{eq:natural extension}
	\psi(y) = \inf_{z\in A} \tilde \psi(z) + Ld_Y(z,y).
	\end{align}
	{\color{black}From \eqref{eq:natural extension} we find $R_C\psi(x) = \inf_{p\in \mathcal P_{d_Y}^1(A)} p(\psi) + C(x,p)$. Indeed, by \cite[Proposition 7.50]{BeSh78} there is for any $\epsilon > 0$ an analytically measurable function $T_\epsilon\colon Y\to A$ with
	$$\tilde \psi(T_\epsilon(y)) + Ld_Y(T_\epsilon(y),y) \leq \psi(y) + \epsilon,$$
	from which {\color{black}after integrating with respect to $p$ and using the definition of the Wasserstein distance} we deduce
	$$p(\psi) - T_\epsilon(p)(\psi) + C(x,p) - C(x,T_\epsilon(p)) {\color{black}\geq -\epsilon+L\mathcal W_1(p,T_\epsilon(p)) + C(x,p) - C(x,T_\epsilon(p)) \geq -\epsilon},$$
	{\color{black}where we used \eqref{def:Lip W 1} in the last inequality.} Therefore, it is actually possible to restrict infimum in $R_C\psi(x)$ to $\mathcal P^1_{d_Y}(A)$, and we conclude
	\begin{align}\label{eq:argument 1}
	\limsup_k R_C\phi_k(x) \leq \inf_{p\in \mathcal P_{d_Y}^1(A)} p(\psi) + C(x,p) = R_C\psi(x).
	\end{align}}
	By dominated convergence, and the fact that $\pi_x\left( A \right)=1$, we have
\begin{align}\label{eq:argument 2}
\lim_k \pi_x(\phi_k) = \pi_x(\psi),
\end{align}
	which yields
\begin{align}\label{eq:argument 3}
0 = \liminf_k \pi_x(\phi_k) + C(x,\pi_x) - R_C\phi_k(x) \geq \pi_x(\psi) + C(x,\pi_x) - R_C\psi(x) \geq 0,
\end{align}
by definition of $R_C\psi(x)$.

%{\color{red}ACA QUEDE}
For $G\subset Y$ define $\Psi_G$ as the set of all $L$-Lipschitz continuous functions on $A$, vanishing at the point {\color{black}$y_0$}, and satisfying
	$$\pi_x(\psi)+C(x,\pi_x)-R_C\psi(x)  = 0\quad \forall x\in G.$$
The previous arguments show that, for each finite $G\subset \Gamma$, the set $\Psi_G$ is nonempty.	
	We now check that $\Psi_G$ is closed in the topology of pointwise convergence:
	Let $(\psi_\alpha)_{\alpha\in \mathcal I}$ be a net in $\Psi_G$ which converges pointwise to a function $\phi$ on $A$. Since $A$ is the countable union of compact sets, it is possible to extract a sequence $(\psi_{\alpha_k})_{k\in\N}$ of the net such that 
	$$\psi_{\alpha_k} \rightarrow \phi \quad \text{pointwise on  $A$ and uniformly on each $K_n$},$$
	from which $\phi$ is $L$-Lipschitz on $A$ and can be extended to an $L$-Lipschitz continuous function $\psi$ on $Y$, see \eqref{eq:natural extension}. By repeating previous arguments (see \eqref{eq:argument 1}, \eqref{eq:argument 2} and \eqref{eq:argument 3}) we obtain that $\phi\in \Psi_G$.
	
	Note that $\Psi_G$ is a closed subset of  $\prod_{y\in A} {\color{black}[-Ld(y,y_0), Ld(y,y_0)]}$ which is compact in the topology of pointwise convergence by Tychonoff's theorem. Further, the collection $\{\Psi_G:\,G\subset\Gamma,|G|<\infty\}$ satisfies the finite intersection property, since if $G_1,\dots,G_n$ are finite then $$\bigcap_{i\leq n}\Psi_{G_i}\supset \Psi_{\cup_{i\leq n}G_i}\neq \emptyset.$$
	Therefore it is possible to find $ \phi \in \bigcap_{G\subseteq \Gamma,~|G|<\infty} \Psi_G$.	Again extend $\phi$, from $A$ to $Y$, by a $L$-Lipschitz function as usual. Thus, we have found the desired dual optimizer.
	\end{proof}

\begin{proof}[Proof of Lemma~\ref{lem:lipschitzduality}] 
By Theorem~\ref{th:kantorovichduality} we have 
\begin{align}\label{eq:gozlankantorovich}
	\inf_{\pi\in \Pi(\mu,\nu)} \int_X C(x,\pi_x)\mu(dx) = \sup_{\phi\in \Phi_{b,1}} \mu(R_C\phi)-\nu(\phi).
\end{align}
By Theorem~\ref{thm existence marginals} we find a minimizer $\pi^*\in \Pi(\mu,\nu)$ of $V(\mu,\nu)$. Now we proceed by taking a maximizing sequence $(\phi_k)_{k\in\mathbb{N}}$ for the right-hand side of \eqref{eq:gozlankantorovich}. Note that we can choose each $\phi_k$, in addition to being below-bounded and continuous, in a way such that it attains its infimum, i.e., there exists $y_k \in Y$ such that
\begin{align}\label{eq:phi attains inf}
-\infty < b_k := \inf_{y\in Y} \phi_k(y) = \phi_k(y_k).
\end{align}
Indeed, this can be done by using e.g.\,$\phi_k \vee \big(b_k + \frac{1}{k}\big)$ instead. Then
$$\lim_k \nu\left(\phi_k - \phi_k \vee \Big(b_k + \frac{1}{k}\Big)\right) = 0,\quad R_C\phi_k \leq R_C\left(\phi_k \vee \Big(b_k + \frac{1}{k}\Big)\right),$$
and the following computation shows that $(\phi_k \vee (b_k + \frac{1}{k}))_{k\in\N}$ is another maximizing sequence:
\begin{align*}
	0 &= \lim_k \int_X [\pi_x^*(\phi_k) + C(x,\pi_x^*) - R_C\phi_k(x)]\mu(dx)\\
	&\geq \lim_k \int_X \Bigg[\pi_x^*\left(\phi_k \vee \Big(b_k +\frac{1}{k}\Big)\right) + C(x,\pi_x^*) - R_C\left(\phi_k \vee\Big(b_k +\frac{1}{k}\Big)\right)(x)\Bigg] \mu(dx) \geq 0.
\end{align*}

So let $\phi_k$ attain its infimum as in \eqref{eq:phi attains inf}. We want to show that we can choose the sequence to be Lipschitz with constant $L$.  For this purpose we infer additional properties of potential minimizers of $R_C\phi_k$. 
Define for each function $\phi_k$ the Borel-measurable sets
	$$A_k := \left\{y \in Y\colon \sup_{y\neq z \in Y} \frac{\phi_k(y)-\phi_k(z)}{d_Y(y,z)}\leq L\right\}\neq \emptyset,$$
	$$\mathcal Y_k := \left\{(y,z) \in Y\times A_k \colon \phi_k(y)-\phi_k(z)> Ld_Y(y,z)\right\}.$$
That $A_k\neq \emptyset $ follows since the minimizers of $\phi_k$ form a subset. We also stress that $$\proj_1(\mathcal Y_k)=A_k^c.$$ Indeed, it is apparent that $\proj_1(\mathcal Y_k) \subset A_k^c$. To see the converse, assume $y\in A_k^c \cap \proj_1(\mathcal Y_k)^c$. Define $Z(z') := \{z\in Y\colon \phi_k(z')-\phi_k(z) > Ld_Y(z,z')\}$. If there exists $ \tilde z \in Z(y)\cap A_k$, we obtain a contradiction to $y\in \proj_1(\mathcal Y_k)^c$. Let $z_0 := y$ and inductively set $z_l \in Z(z_{l-1})$ such that
\begin{align}\label{eq:sequence inequality}
	\inf_{z\in Z(z_{l-1})} \phi_k(z) + \frac{1}{2^l} \geq \phi_k(z_l).
\end{align}
We have for any natural numbers $0\leq i< n$
\begin{align}\label{eq:sequence Z condition}
\phi_k(z_i) - \phi_k(z_n) = \sum_{l=i}^n \phi_k(z_{l-1}) - \phi_k(z_l)> L\sum_{l=i}^n d_Y(z_{l-1},z_l).
\end{align}

The r.h.s.\ is bounded from below by $Ld_Y(z_i,z_n)$ and so as before we see that $z_n\in A_k$ provides a contradiction. We therefore assume for all $l$ that $z_l\notin A_k$. The above inequality yields by lower-boundedness of $\phi_k$ that $(z_l)_{l\in\N}$ is a Cauchy sequence in $Y$. Writing $\bar z$ for its limit point, we conclude from \eqref{eq:sequence Z condition} that $\phi_k(z_i) - \phi_k(\bar z) > Ld_Y(z_i,\bar z)$ and consequentely $Z(\bar z) \subset Z(z_i)$. Since then $\inf\{\phi_k(z):z\in Z(z_i)\}\leq \inf\{\phi_k(z):z\in Z(\bar z)\}$ and from \eqref{eq:sequence inequality}, we deduce $\inf\{\phi_k(z):z\in Z(\bar z)\}\geq \phi_k(\bar z)$. Thus $Z(\bar z)=\emptyset$, implying $\bar z\in A_k$ and yielding a contradiction to $y\in \proj_1(\mathcal Y_k)^c$. All in all, we have proven that $A_k^c = \proj_1(\mathcal Y_k)$.

By Jankov-von Neumann uniformization \cite[Theorem 18.1]{Ke95} there is an analytically measurable selection $T_k\colon \proj_1(\mathcal Y_k) \rightarrow A_k$.  We set $T_k$ on $A_k=\proj_1(\mathcal Y_k)^c$ as the identity. 
Then $T_k$ maps from $Y$ to $A_k$ and for any $p\in \mathcal P_{d_Y}^t(Y)$ we have
\begin{align*}
	C(x,T_k(p)) & \leq C(x,p)+L\mathcal W_1(p,T_k(p)) \\
	& \leq C(x,p)+L \int_Y d_Y(y,T_k(y))p(dy)\\
	& \leq C(x,p)+ \int_Y [\phi_k(y)-\phi_k(T_k(y))]p(dy)\\
	&= C(x,p)+p(\phi_k) - T_k(p)(\phi_k).
\end{align*}
Therefore, we can assume that potential minimizers of $R_C\phi_k$ are concentrated on $A_k$:
\begin{align}\label{eq:RCrestrict}
	R_C\phi_k(x) = \inf_{p\in\mathcal P_{d_Y}^1(Y)} p(\phi_k) + C(x,p)
	= \inf_{p\in\mathcal P_{d_Y}^1(A_k)} p(\phi_k) + C(x,p).
\end{align}
{\color{black}We introduce a family of $L$-Lipschitz continuous functions by 
	$$\psi_k(y) := \inf_{z\in A_k} \phi_k(z) + Ld_Y(y,z) = \inf_{z\in Y} \phi_k(z) + Ld_Y(y,z)\quad \forall y\in Y,$$
where equality holds thanks to $\proj_1(\mathcal Y_k)=A_k^c$, since for $z\in A_k^c$ we find $(z,\hat z) \in \mathcal Y_k$, and so
$$\phi_k(z) + Ld_Y(y,z) > \phi_k(\hat z) + L(d_Y(y,z) + d_Y(z,\hat z)) \geq \phi_k(\hat z) + Ld_Y(y,\hat z).$$}
Then $\phi_k \geq \psi_k$ where equality holds precisely on $A_k$. Similarly to before, we find a measurable selection $\hat T_k\colon Y \rightarrow A_k$ such that $\psi_k(\hat T_k(y)) + Ld_Y(y,\hat T_k(y)) \leq \psi_k(y) + \epsilon$. For any $p \in \mathcal P^t_{d_Y}(Y)$ we have
\begin{align*}
	C(x,\hat T_k(p)) \leq C(x,p) + L\int_Y d_Y(y,\hat T_k(y))p(dy) \leq C(x,p) + p(\psi_k) - \hat T_k(p)(\psi_k) + \epsilon.
\end{align*}
Since $\epsilon$ is arbitrary, by the same argument as in \eqref{eq:RCrestrict}, we can restrict $\mathcal P_{d_Y}^1(Y)$ to $\mathcal P_{d_Y}^1(A_k)$ in the definition of $R_C\psi_k$. Hence, $R_C\phi_k(x) = R_C\psi_k(x)$ and
	\begin{align*}
	\int_X C(x,\pi^*_x)  \mu (dx) &= \lim_k \int_X \left [-\pi^*_x(\phi_k)+ R_C\phi_k(x)\right ]\mu(dx)\\
	&\leq \lim_k \int_X \left [-\pi^*_x(\psi_k) + R_C\psi_k(x)\right ]\mu(dx)\\
	&\leq \lim_k\int_X \left [-\pi^*_x(\psi_k)+ \pi^*_x(\psi_k)+C(x,\pi^*_x)\right ]\mu(dx) \\ &= \int_X C(x,\pi^*_x)\mu(dx).
	\end{align*}	
\end{proof}

\section{On the Brenier-Strassen Theorem of Gozlan and Juillet}
\label{sec BS GJ}

In this part we take $X=Y=\mathbb R^d$, equipped with the Euclidean metric, and $$C_{\theta}(x,\rho):= \theta\left( x - \int y\rho(dy) \right ),$$ where $\theta:\mathbb R^d\to\mathbb R_+$ is convex. As usual we denote by $V(\cdot,\cdot)$ the value of the weak transport problem with this cost functional (see \eqref{eq weak transport def}). We have

\begin{lemma}\label{lem basic equality theta}
Let $\mu\in \mathcal P(\mathbb R^d)$ and $\nu\in \mathcal P^1(\mathbb R^d)$. Then
\begin{equation}\label{eq GJ theta}
\inf_{\eta\leq_c\nu}\inf_{\pi\in\Pi(\mu,\eta)}\int \theta(x-z)\pi( dx, dz)=V(\mu,\nu).
\end{equation}
\end{lemma}

\begin{proof}
Given $\pi$ feasible for $V(\mu,\nu)$, we define $T(x):=\int y\pi^x(dy)$ and notice that $T(\mu)\leq_c \nu$ by Jensen's inequality. From this we deduce that the l.h.s.\ of \eqref{eq GJ theta} is smaller than the r.h.s. For the reverse inequality, let $\epsilon > 0$ and say $\bar\eta\leq_c\nu$ is such that $$\inf_{\eta\leq_c\nu}\inf_{\pi\in\Pi(\mu,\eta)}\int \theta(x-z)\pi( dx, dz) + \epsilon \geq \inf_{\pi\in\Pi(\mu,\bar\eta)}\int \theta(x-z)\pi( dx, dz)\geq \int \theta(x-z)\bar\pi( dx, dz)-\epsilon,$$ for some $\bar \pi\in \Pi(\mu,\bar\eta)$. By Strassen theorem there is a martingale measure $m( dz, dy)$ with first marginal $\bar\eta$ and second marginal $\nu$. Define $\pi( dx, dy):=\int_{z}\bar \pi^{z}( dx)m^{z}( dy)\bar\eta( dz)$, so then $\pi$ has x-marginal $\mu$ and y-marginal $\nu$, and furthermore $\int y\pi^ x( dx)=\int z \bar\pi^x( dx)$ ($\mu$-a.s.), by the martingale property of $m$. Thus, by Jensen's inequality:
$$\int \theta(x-z) \bar\pi_x(dz)\mu(dx)\geq \int\theta\left(x-\int z \bar\pi_x(dz)\right)\mu(dx)=\int\theta\left(x-\int y \pi_x(dy)\right)\mu(dx)\geq V(\mu,\nu).$$
Taking $\epsilon\to 0$ we conclude.
\end{proof}

We now provide the proof of Theorem \ref{thm a la GJ}, in which case $\theta(\cdot)=|\cdot|^2$:

\begin{proof}[Proof of Theorem \ref{thm a la GJ}]
We have $V(\mu,\nu)<\infty$, since the product coupling yields a finite cost. Lemma \ref{lem basic equality theta} established the rightmost equality in \eqref{eq GJ mu star}. The existence of an optimizer $\pi$ to $V(\mu,\nu)$ follows from Theorem \ref{thm existence marginals}. By the necessary monotonicity principle (Theorem \ref{thm cyclical}) there exists a measurable set $\Gamma\subseteq X$ with $\mu(\Gamma)=1$ such that for any finite number of points $x_1,\dots,x_N$ in $\Gamma$ and measures $m^1,\dots,m^N$ in $\mathcal P(\mathbb R^d)$
	with $\sum_{i=1}^N m^i = \sum_{i=1}^N \pi^{x_i}$ the following inequality holds:
	\begin{align}\label{eq ours cm}
		\sum_{i=1}^N \left|x^i-\int y\pi^{x^i}(dy)\right |^2 \leq \sum_{i=1}^N \left|x^i-\int ym^i(dy)\right |^2.
	\end{align}
In particular, if we let  $$T(x):=\int y\pi_x(dy),$$ and $\sigma$ is any permutation, then 
\begin{align}\label{eq ours cm2}
		\sum_{i} \left|x^i-T(x^i)\right |^2 \leq \sum_{i=1}^N \left|x^i-T(x^{\sigma(i)})\right |^2.
	\end{align}
	Let us introduce $p(dx,dz):=\mu(dx)\delta_{T(x)}(dz)$ and observe that its $z$-marginal is $T(\mu)$. By Rockafellar's theorem (\cite[Theorem 2.27]{Vi03}) the support of $p$ is contained in the graph of the subdifferential of a closed convex function. Then by the Knott-Smith optimality criterion (\cite[Theorem 2.12]{Vi03}) the coupling $p$ attains $\mathcal W_2(\mu,T(\mu))$. Since by Jensen clearly $T(\mu)\leq_c\nu$, this establishes the remaining equality in \eqref{eq GJ mu star} and shows further that $V(\mu,\nu)=\mathcal W_2(\mu,T(\mu))^2$ and $\mu^*:=T(\mu)$. The uniqueness of $\mu^*$ follows the same argument as in the proof of \cite[Proposition 1.1]{GoJu18}. 
	
	We can use \eqref{eq ours cm} and argue verbatim as in \cite[Remark 3.1]{GoJu18} showing that $T$ is actually 1-Lipschitz on $\Gamma$. We will now prove that $T$ is ($\mu$-a.s.\ equal to) the gradient of a continuously differentiable convex function. The key remark is that the coupling $p$ is also optimal for $V(\mu,T(\mu))$. Indeed, we have
	$$V(\mu,\nu)\leq \inf_{\eta\leq_c T(\mu)}\mathcal W_2(\mu,\nu)^2 = V(\mu,T(\mu))\leq \int |x-T(x)|^2\mu(dx)=V(\mu,\nu). $$ 
{\color{black}Take any $\mathcal W_2$-approximative sequence $(\mu^k)_{k\in\mathbb N}$ of $\mu$ such that for all $k\in\N$
$$\mu^k \ll \lambda \ll \mu^k,$$
where $\lambda$ denotes the $d$-dimensional Lebesgue measure. This can be easily achieved by scaled convolution with a non-degenerate Gaussian kernel. By stability of the considered weak transport problem \cite[Theorem 1.5]{BaBePa19}, and using the previously shown, we obtain for each $\mu^k$ a 1-Lipschitz map $T^k$ defined this time everywhere in $\R^d$ with
$$\mathcal W_2(\mu^k,T^k(\mu^k))^2 = V(\mu^k,\nu),$$
and $T^k(\mu^k) \to T(\mu)$ in $\mathcal W_1$.
By Brenier's theorem \cite[Theorem 2.12 $(ii)$]{Vi03} we find for each $k\in\N$ some convex function $\phi^k\colon\R^d \to \R$, $\phi(0) = 0$, and $\nabla \phi^k(x) = T^k(x)$ $\lambda$-a.e. $x$. By continuity of $T^k$ we have $\nabla\phi^k(x) = T^k(x)$ for all $x\in \R^d$.

We want to show that $(\phi^k)_{k\in\N}$ is suitably relatively compact. {\color{black}By tightness of $\mu^k$ and $T^k(\mu^k)$ we find compact sets $K_1,K_2\subset\R^d$ with
$$\quad \inf_k \mu^k(K_1) > \frac{1}{2},\quad \inf_k T^k(\mu^k)(K_2) > \frac{1}{2}.$$
In particular, the sets $(T^k(K_1) \cap K_2)_{k\in\N}$ are all non-empty. The compactness of $K_1$ and $K_2$, and the 1-Lipschitz property of each $T^k$, imply then the existence of $x\in K_1$ such that $\sup_k |T^k(x)| < \infty$.} Hence, $(T^k)_{k\in\N}$ is pointwise bounded and uniformly 1-Lipschitz. Thanks to Arzel\`a-Ascoli's theorem and a diagonalization argument, we can select a subsequence $(T^{k_j})_{j\in\N}$ of $(T^k)_{k\in\N}$ which converges locally uniformly to some 1-Lipschitz function $\tilde T\colon\R^d\to\R^d$. Since being a gradient field is preserved under locally uniform limits, we have that $\tilde T$ is a gradient field, and $\phi^{k_j}$ converges pointwise to some $\phi$ with $\phi(0) = 0$ and $\nabla \phi = \tilde T$. In particular $\phi$ is convex and of class $C^1(\R^d)$. 

Finally, for any $f \in C_b(\R^d)$ and $\epsilon > 0$, we find an index $j_0 \in \N$ such that for all $j \geq j_0$:
\begin{align*}
	|T^{k_j}(\mu^{k_j})(f) - \tilde T(\mu)(f)| &\leq 
	|T^{k_j}(\mu^{k_j})(f) - \tilde T(\mu^{k_j})(f)| + |\tilde T(\mu^{k_j})(f) - \tilde T(\mu)(f)| < \epsilon,
\end{align*}
where the first summand can be chosen sufficiently small for large $j$ by locally uniform convergence of $T^{k_k}$ to $\tilde T$ and the second one by weak convergence of $\mu^{k_j}$ to $\mu$. All in all, we deduce that $T^{k_j}(\mu^{k_j})$ converges weakly to $\tilde T(\mu)$, which must therefore match $T(\mu)$. Furthermore, $\mu(dx)\delta_{\tilde T(x)}(dy)$ defines an optimizer for the weak transport problem \eqref{eq weak transport def} between $\mu$ and $\nu$ with cost \eqref{GJcost}. By uniqueness of the optimizers we conclude $T = \tilde T$ $\mu$-almost surely. In particular, $T$ is $\mu$-almost everywhere the gradient of the convex function $\phi\in C^1(\R^d)$.
}
\end{proof}

\bibliography{joint_biblio}{}
\bibliographystyle{abbrv}
\end{document}